\documentclass[11pt]{amsart}

\usepackage{mathtools, amssymb, amsfonts, bbm}
\usepackage{amsthm}
\usepackage{thmtools}
\usepackage[margin=1.0in]{geometry}
\usepackage{enumerate}

\numberwithin{equation}{section}

\usepackage{hyperref}
\usepackage[capitalize]{cleveref}

\newcommand{\1}{\mathbbm{1}}
\newcommand{\R}{\mathbb{R}}

\newcommand{\Q}{\mathbb{Q}}
\newcommand{\C}{\mathbb{C}}
\newcommand{\Z}{\mathbb{Z}}
\newcommand{\N}{\mathbb{N}}
\newcommand{\cB}{\mathcal{B}}
\newcommand{\cD}{\mathcal{D}}

\renewcommand{\phi}{\varphi}
\newcommand{\grad}{\nabla}
\DeclareMathOperator{\Cay}{Cay}
\DeclareMathOperator{\Sch}{Sch}
\DeclareMathOperator{\Orth}{O}
\DeclareMathOperator{\supp}{supp}
\DeclareMathOperator{\im}{im}
\newcommand*{\EE}{
  \mathop{
    \mathchoice{\vcenter{\hbox{\larger[4]$\mathbb{E}$}}}
               {\kern0pt\mathbb{E}}
               {\kern0pt\mathbb{E}}
               {\kern0pt\mathbb{E}}
  }\displaylimits
}

\DeclarePairedDelimiter{\abs}{\lvert}{\rvert}
\DeclarePairedDelimiter{\norm}{\lVert}{\rVert}
\DeclarePairedDelimiter{\ip}{\langle}{\rangle}

\newcommand{\st}{\colon}

\renewcommand{\b}[1]{\left( #1 \right)}

\theoremstyle{plain}
\newtheorem{theorem}{Theorem}[section]
 
\newtheorem{lemma}[theorem]{Lemma}       
\newtheorem{prob}[theorem]{Problem}
\newtheorem{prop}[theorem]{Proposition}

\newtheorem{obs}[theorem]{Observation}

\theoremstyle{definition}
\newtheorem{defn}[theorem]{Definition}
\newtheorem{exa}[theorem]{Example}

\newtheorem{claim}{Claim}[lemma]
\newenvironment{proofofclaim}{\begingroup\begin{proof}[Proof of claim]}{\end{proof}\endgroup}

\makeatletter

\newcommand{\deferproof}[2]{%
  \expandafter\xdef\csname defer@#1\endcsname{#2}%
  \@for\i:=#2\do{%
    \expandafter\xdef\csname defer@#1@\i\endcsname{\number\value{\i}}%
  }%
}

\newcommand{\restorecounters}[1]{%
  \edef\deferlist{\csname defer@#1\endcsname}%
  \@for\i:=\deferlist\do{%
    \setcounter{\i}{\csname defer@#1@\i\endcsname}%
  }%
}

\makeatother

\newenvironment{proofwithclaims}[1]{%
    \begingroup
    \setcounter{claim}{0}%
    \renewcommand{\theclaim}{\getrefnumber{#1}.\arabic{claim}}%
    \def\theHclaim{proof.#1.\arabic{claim}}%
    \begin{proof}[Proof of \cref{#1}]
}{%
    \end{proof}\endgroup
}

\title{An inverse problem on eigenfunction triple products}
\author{Carl Schildkraut and Romain Speciel}
\date{\today}

\begin{document}

\begin{abstract}
	On a connected closed smooth Riemannian manifold, the algebraic structure of the Laplace eigenfunctions, as described by eigenfunction triple products, uniquely determines the geometry. We refine this correspondence by introducing the notion of an $N$-product eigenbasis, which consists of eigenfunctions whose pairwise products may be written as linear combinations of at most $N$ basis elements. We prove that a manifold admits a $2$-product eigenbasis if and only if it is a flat torus. We also prove an analogous result for Laplace eigenvectors of bounded-degree graphs. 
\end{abstract}

\maketitle

\section{Introduction}
\label{sec: intro}

Let $(M,g)$ be a connected closed smooth Riemannian manifold. The metric naturally gives rise to the Laplacian $\Delta=-\textrm{div}\, \grad$, a differential operator whose eigenspaces are pairwise orthogonal and span $L^2(M)=L^2(M, \mathbb{R})$. Fix an orthonormal basis $\{\phi_i\}_{i=0}^\infty$ of $L^2(M)$ consisting of Laplace eigenfunctions ordered such that the corresponding eigenvalues $\lambda_i\geq 0$ are nondecreasing in $i$. Famously, the spectrum $\{\lambda_i\}$ of the Laplacian determines neither the metric $g$ nor the manifold $M$. Indeed, there exist pairs of non-isometric, and even non-homeomorphic, Riemannian manifolds that share the same Laplace spectrum. Consult \cite{Gordon2000} for a thorough survey of the subject. These examples demonstrate that one cannot recover $M$ from the linear map $\Delta\colon C^\infty(M)\to C^\infty(M)$ alone, viewing $C^\infty(M)$ here as an abstract real vector space.

However, in addition to being a vector space, $C^\infty(M)$ admits an algebra structure. From the perspective of an eigenbasis, the multiplication may be encoded by the triple product constants:

\begin{defn}
	The \textit{triple product constants} (or \textit{structure constants}) corresponding to an orthonormal basis $\{\phi_i\}_{i=0}^\infty$ of $L^2(M)$ are the coefficients $c_{ijk}$ given by
	\[
		c_{ijk}=\int_M\phi_i\phi_j\phi_k \,d\mathrm{vol}_g,
	\]
	where $d\mathrm{vol}_g$ denotes the volume form associated to the metric $g$.
\end{defn}

\noindent Since
\[
	\phi_i\phi_j=\sum_k c_{ijk}\,\phi_k
\]
by orthonormality, the multiplicative structure of $C^\infty(M)$ is encoded by the triple product constants. The triple product constants have been studied in a range of contexts spanning from number theory to geometric analysis. We provide at the start of \cref{sec: examples} an account of this past work.

While $M$ cannot be recovered from its spectrum alone, the triple product constants provide enough additional information to complete the reconstruction.

\begin{theorem}
\label{thm: recover M from triple prod}
    $(M,g)$ is uniquely determined, up to isometry, by its spectrum together with the triple product constants in some orthonormal eigenbasis.
\end{theorem}

\noindent This result appears in \cite[Remark 4.3]{Connes2008}, see also \cite{Schaefer2026}. We record a simple proof in Appendix \ref{sec: reconstruction}. \cref{thm: recover M from triple prod} suggests the following geometric inverse problems: how can one read off the shape of $M$ from the triple product constants? How does the algebraic structure of the eigenfunctions connect to the geometry of the underlying space? In this paper, we provide some results towards answering these questions. To this end, we introduce the following definitions to describe the multiplicative properties of a Laplace eigenbasis $\{\phi_i\}_{i=0}^\infty$ of $L^2(M)$.

\begin{defn}
     Call $\{\phi_i\}$ a \textit{finite product eigenbasis} if the product of any two basis elements may be written as a finite linear combination of basis elements. Equivalently, $\{\phi_i\}$ is a finite product eigenbasis if its algebraic span forms an algebra under pointwise multiplication.
\end{defn}

\begin{defn}
	  Call $\{\phi_i\}$ an \textit{$N$-product eigenbasis} if the product of any two basis elements may be written as a linear combination of at most $N$ basis elements.
\end{defn}

\noindent If a manifold admits a finite product eigenbasis, then all its Laplace eigenbases will be finite product bases. However, the same is not true for $N$-product bases. Indeed, on a manifold which admits an $N$-product eigenbasis, another poorly chosen eigenbasis could have products which require arbitrarily many terms to be written as a sum. (One example is the square torus.) We provide background and motivation for these definitions, and outline several key examples, in \cref{sec: examples}.

Our first main result is a complete characterization of flat tori in terms of the algebraic structure of their eigenfunctions. We establish the following theorem.

\begin{theorem}
\label{thm: 2 product eigenbasis iff flat torus}
	$(M,g)$ admits a 2-product eigenbasis if and only if it is a flat torus.
\end{theorem}

\noindent The proof proceeds in two steps. To begin, in \cref{sec: 2PP implies flat}, we exploit the structure of 2-product bases to compute the curvature of the underlying manifold and show

\begin{prop}
\label{prop: 2 product eigenbasis implies flatness}
	If $(M,g)$ admits a 2-product eigenbasis, then it is flat.
\end{prop}

\noindent The main technical step consists of extracting the structure of the nodal set of the eigenfunctions in a 2-product eigenbasis, and exploiting this structure to build high frequency eigenfunctions whose gradients align in a single direction. These gradients are then used to construct local frames from which we deduce flatness.

In \cref{sec:graph,sec:graph-detail}, we give a second proof of \cref{prop: 2 product eigenbasis implies flatness} which is less geometric in nature.

In \cref{sec: flat manifolds and the NPP}, we study flat manifolds and show that their algebraic structure corresponds to the smallest $N$ such that $M$ admits an $N$-product eigenbasis. More precisely, we obtain

\begin{prop}
\label{prop: n product bases and holonomy}
	Let $(M,g)$ be a flat manifold, and set 
    \[
    \nu(M)=\inf\{N : \textup{$M$ admits an $N$-product eigenbasis}\}.
    \]
    Denoting the holonomy group corresponding to $M$ by $H$, we have
	\[
	    2\abs{H}\leq\nu(M) \leq 4\abs{H}.
	\]
\end{prop}

\noindent This result is obtained by carefully studying the multiplicative properties of eigenfunctions of Bieberbach manifolds. \cref{thm: 2 product eigenbasis iff flat torus} follows immediately after combining Propositions~\ref{prop: 2 product eigenbasis implies flatness}~and~\ref{prop: n product bases and holonomy}.

\subsection{Graphs}

There is a well-studied analogy between the spectral theory of Riemannian manifolds and that of bounded-degree graphs. 
To give a handful of references: graphs constructed by discretization procedures can be used to study the Laplace eigenvalues of manifolds as in \cite{Buser, Mantuano}; problems on manifolds can be directly reduced to problems on graphs, or vice versa, as in \cite{AminiCohenSteiner, BuserCubic}; and examples \cite{ColboisColin, ColboisGirouard} and techniques \cite{AnantharamanMonk, LetrouitMachado} can be transferred from graphs to manifolds.

Given a graph $G$, we define its \emph{adjacency matrix} $A_G\in\R^{V(G)\times V(G)}$ such that the $uv$ entry of $A_G$ is $1$ if $u$ and $v$ are connected by an edge and $0$ otherwise. Its \emph{degree matrix} $D_G\in\R^{V(G)\times V(G)}$ is a diagonal matrix whose $vv$ entry is the degree of $v$ in $G$, and its \emph{Laplacian matrix} is $\Delta_G=D_G-A_G$. Like in the manifold setting, $\Delta_G$ is self-adjoint and positive semi-definite. The constant function is an eigenfunction\footnote{To keep terminology mostly consistent between manifolds and graphs, we will generally identify vectors in $\R^{V(G)}$ with functions $V(G)\to\R$, and speak of eigenvectors as eigenfunctions.} of $\Delta_G$ with eigenvalue zero, and if $G$ is connected all other eigenvalues are strictly positive. Enumerating an eigenbasis $\phi_1,\ldots,\phi_n$ for $\Delta_G$, we can define the triple product constants
\[c_{ijk}=\frac1{\abs{V(G)}}\sum_{v\in V(G)}\phi_i(v)\phi_j(v)\phi_k(v).\]
These triple product constants display the same reconstruction result as \cref{thm: recover M from triple prod}:

\begin{theorem}
\label{thm: graph reconstruction}
    A graph $G$ is determined up to isomorphism by its Laplace spectrum together with the triple product constants $\{c_{ijk}\}_{i,j,k=1}^n$ in some orthonormal eigenbasis.\footnote{The proof of \cref{thm: graph reconstruction} we present works for weighted graphs as well, but we only focus on simple unweighted graphs for the remainder of the article.}
\end{theorem}

We will need to define a few graph classes of interest to state our main result and surrounding examples in the graph setting.

\begin{defn}
    Let $\Gamma$ be a finitely generated group and let $S\subset \Gamma\setminus\{\operatorname{id}\}$ be a generating set satisfying $s^{-1}\in S$ for all $s\in S$. 
    \begin{itemize}
        \item The \emph{Cayley graph} $\Cay(\Gamma,S)$ of $\Gamma$ with generating set $S$ is the graph with vertex set $\Gamma$ and edge set $\{(g,gs):g\in \Gamma,s\in S\}$.

        \item If $H$ is a subgroup of $\Gamma$, the \emph{Schreier coset graph} $\Sch(\Gamma,H,S)$ is the graph with vertex set $\{Hg:g\in \Gamma\}$ and edges connecting $Hg$ to $Hgs$ for each $s\in S$.

        \item If $\Gamma$ is abelian and $T\subset \Gamma$, the \emph{Cayley sum graph} $\Cay^+(\Gamma,T)$ is the graph with vertex set $\Gamma$ and edge set $\{(x,t-x):x\in \Gamma,t\in T\}$. Since we are primarily concerned with simple graphs, we remove any self-loops which arise.\footnote{This construction produces an undirected graph even if $T$ is not symmetric.}
    \end{itemize}
    Cayley sum graphs and Schreier coset graphs potentially have self-loops and multi-edges.
\end{defn}

Cayley graphs on abelian groups can perhaps be considered natural analogues of tori; like tori, they are acted on simply transitively by an abelian group of ``translation'' automorphisms.
One cannot think of Cayley sum graphs in quite the same way, but they can be represented as Schreier coset graphs on ``nearly abelian'' (generalized dihedral) groups, as follows.
Given an abelian group $\Gamma$, form a semidirect product $\Gamma\rtimes\Z/2\Z$ where the $\Z/2\Z$ factor acts on $\Gamma$ by negation. Writing $r$ for the unique non-identity element of the subgroup $\Z/2\Z$, we have
\[\Cay^+(\Gamma,T)\cong\Sch(\Gamma\rtimes\Z/2\Z,\Z/2\Z,rT).\]
(Note that every element in $r\Gamma$ is an involution in $\Gamma\rtimes\Z/2\Z$, so the set $rT$ is symmetric for any $T\subset\Gamma$.)
In this sense, Cayley sum graphs loosely analogize some flat manifolds with holonomy group of order $2$.

Our second main result analogizes \cref{thm: 2 product eigenbasis iff flat torus}. 

\begin{theorem}
\label{thm: 2PP iff Cayley}
    For every positive integer $d$, there exists some positive integer $f(d)$ such that the following holds:

    Let $G$ be a connected graph with maximum degree at most $d$ on at least $f(d)$ vertices. Suppose that $G$ admits a $2$-product eigenbasis and also that $\abs{V(G)}$ is odd.
    Then $G$ is either a Cayley graph, or a Cayley sum graph with symmetric generating set, on an abelian group.
\end{theorem}

The proof of \cref{thm: 2PP iff Cayley} is composed of three stages.
Much of the first stage can be\footnote{For reasons of clarity, we often restrict ourselves to the two settings of interest, as this simplifies some of the proof ingredients.} run in the generality of an arbitrary orthonormal $2$-product eigenbasis $\cB$ of the space of $L^2$ functions over an arbitrary (finite or infinite) probability space $(X,\mu)$.
This generality covers both the graph setting discussed here and the problem in manifolds discussed earlier in this section.
The idea is to define, for each $\phi\in\cB$, an auxiliary graph whose vertex set is the set $\cB$ and where two vertices $\psi_1$ and $\psi_2$ are connected by an edge if the triple product constant corresponding to $(\phi,\psi_1,\psi_2)$ is nonzero.
Studying the structure of this graph allows us to understand both the structure of our chosen $\phi\in\cB$ and the relationship between $\phi$ and other elements of $\cB$.
We state the relevant lemmas in \cref{sec:graph-aux} and prove them in \cref{sec:graph-detail}.

Since this first stage applies also in the manifold setting, we use the results we obtain here to, in \cref{sec:comp-manifold}, provide a second proof of \cref{prop: 2 product eigenbasis implies flatness}. 

This first stage requires a technical assumption: namely, the chosen eigenfunction $\phi$ must take sufficiently many distinct values.
In the second stage of the proof of \cref{thm: 2PP iff Cayley}, we show that such an assumption holds for some such $\phi$ if $\abs{V(G)}$ is large in terms of the maximum degree of $G$. (It turns out that such a statement is sufficient to obtain a similar, albeit quantitatively weaker, property for enough other eigenfunctions $\psi$ in a $2$-product eigenbasis.)
Such a conclusion is potentially of independent interest, and does not require any algebraic conditions on the basis, so we state it here.

\begin{prop}\label{prop:find-varied}
    Let $d,k,n$ be positive integers, and suppose that
    \[d<(\log\log n)^{1/(2k)}.\]
    Let $G$ be a connected graph on $n$ vertices of maximum degree at most $d$.
    Every eigenbasis of $G$ contains an eigenfunction which takes at least $k$ distinct values.
\end{prop}
\deferproof{prop:find-varied}{section,theorem}

The last stage in the proof is to use the information extracted in the first stage to construct an associated abelian group $\Gamma$ with $\abs{\Gamma}=\abs{V(G)}$. 
This $\Gamma$ is the dual group of a group whose elements are constructed directly from the eigenfunctions $\varphi$.
We then prove that $G$ is a Cayley graph or Cayley sum graph on $\Gamma$. 

\subsection{Future work}

These results suggest several further lines of investigation, which we record below. A natural extension of \cref{prop: 2 product eigenbasis implies flatness} asks if a manifold which admits an $N$-product eigenbasis, with $N\geq 3$, must be flat. This question is answered affirmatively by a paper of Josef Greilhuber and the second author \cite{GreilhuberSpeciel}, using a different set of techniques than those presented here. However, more generally, admitting a finite product eigenbasis seems indicative of an underlying algebraic structure. A simple extension of the argument to be presented in \cref{ex: Lie groups}, which also appears as \cite[Lemma~9]{LinMendesRadeschi2025}, shows that all manifolds which are finitely covered by a homogeneous space (i.e., a space with transitive isometry group) admit a finite product eigenbasis. However, complete classification seems subtle, and we pose this as an open problem. 

\begin{prob}\label{prob:fin-prod}
	Classify all manifolds $(M,g)$ which admit finite product bases.
\end{prob}

\noindent Note that this question appears in a different context as Question~3 in \cite{LinMendesRadeschi2025}.

In the graph setting, there seems not to be a natural analogue of a finite product eigenbasis, but we may still pose an analogue of the flatness question in the presence of an $N$-product eigenbasis.

\begin{prob}\label{prob:N-prod-graph}
    Let $G$ be an undirected simple connected graph admitting an $N$-product eigenbasis.
    Suppose that $\abs{V(G)}$ is sufficiently large in terms of both $N$ and the maximum degree of $G$.
    Show that $G$ is a Schreier coset graph on some group $\Gamma$ with an abelian normal subgroup of index at most $F(N)$.
\end{prob}

\noindent See \cref{ex:cayley-gen} for some explanation of the graphs described in \cref{prob:N-prod-graph}.

There is also the more modest question of removing the additional assumptions in \cref{thm: 2PP iff Cayley}, which may be tractable. (See \cref{ex:graph-join} for an explanation of some of the intricacies of this question.)

\begin{prob}\label{prob:2-prod-graph}
    Classify the connected graphs admitting a $2$-product eigenbasis.
\end{prob}

\vspace{1em}

\subsection*{Acknowledgements} The authors thank Benjamin Church, Josef Greilhuber, Eva-Maria Hekkelman, and Zeev Rudnick for helpful conversations, and Jacob Fox and Rafe Mazzeo for their guidance. CS was supported by a NSF Graduate Research Fellowship Program under Grant No.~DGE-2146755.

\subsection*{AI Declaration} 
The vast majority of the math in this paper was done without AI assistance.
We used GPT-6 Astra to check some of our examples and conjectures: in particular, we used it to check that \cref{prob:N-prod-graph} seemed plausible and to find \cref{exa: twisted torus} (which arose from investigation of an error in a draft version of this paper). Later, we used GPT-6 Astra to proofread the paper and suggest edits. This contributed to the content of \cref{lem:eigsq} and substantially simplified the proof of \cref{lem:no-zero-product}.

\vspace{3em}

\section{Background, Motivation, and Examples}
\label{sec: examples}

Eigenfunction triple products were originally studied in number theory due to their connection to the theory of $L$-functions. For example, Watson's formula explicitly relates the triple product constants on arithmetic hyperbolic surfaces to special values of $L$-functions \cite{Watson2008}. Work by Sarnak \cite{Sarnak1994}, Bernstein and Reznikov \cite{BernsteinReznikov1999, BernsteinReznikov2004}, and Reznikov \cite{Reznikov2001} establishes upper and lower bounds on the exponential decay of triple products in this context. These have since been employed to conclude results on subconvexity for $L$-functions \cite{BernsteinReznikov2010} and on quantum unique ergodicity \cite{BisainHumphriesMandelshtamWalshWang2024}.

In the analytic and geometric setting, the study of eigenfunction triple products was initiated by Zelditch in \cite{Zelditch2012}. In the case of compact real-analytic manifolds, he obtained exponentially decaying bounds on the triple product constants. These bounds, given in terms of the maximal analytic Grauert tube radius, mirror those which appear in the number theoretic setting. They have since been refined, as in \cite{CharronPagano2025}, for example. Beyond the real-analytic setting, work by Steinerberger \cite{Steinerberger2019} investigates the spectral resolution of products of eigenfunctions. Lu, Sogge, and Steinerberger \cite{LuSoggeSteinerberger2019} prove that products of low-frequency eigenfunctions can be efficiently approximated by a comparatively low-dimensional spectral subspace. In other words, triple product constants are small outside of a controllable finite window. Wyman \cite{Wyman2022} refined this picture by showing that the constants must concentrate, in the $\ell^2$ sense, around a model distribution.

As discussed in \cref{sec: intro}, our work addresses the inverse direction to these results. Instead of assuming structure on the underlying space, such as real-analytic or arithmetic, and describing the behavior of the triple product constants, we attempt to recover information about the manifold from the triple product constants. To motivate this approach, we discuss below three examples: tori (which admit $N$-product bases), compact Lie groups (which admit finite product bases), and hyperbolic surfaces (which do not admit finite product bases).

\begin{exa}[Tori]
\label{ex: torus}

	Let $\Lambda\subset \mathbb{R}^n$ be a full rank lattice and set $M=\mathbb{R}^n/\Lambda$ to be the corresponding flat torus. In the complex Fourier basis $e_\xi(x)=e^{2\pi i\langle \xi,x\rangle}$, with $\xi$ in the dual lattice $\Lambda^*$, the product of two basis elements is exactly a basis element. Passing to the associated real basis of sines and cosines, the product of any two basis elements lies in the span of at most two basis elements. Since these are all eigenfunctions and they span $L^2(M)$, we conclude that flat tori admit a $2$-product eigenbasis.

	This phenomenon depends on the choice of basis. Indeed, not all eigenbases on the torus are $2$-product bases. On highly degenerate tori, such as the one corresponding to the square lattice, the eigenspaces can have arbitrarily large multiplicity. By choosing an orthonormal basis within each eigenspace which mixes many Fourier modes, products of the basis eigenfunctions may be arranged to spread across a growing number of eigenspaces.
\end{exa}

\begin{exa}[Compact Lie groups]
\label{ex: Lie groups}

Let $G$ be a compact Lie group equipped with a bi-invariant Riemannian metric. We show that $G$ admits a finite product eigenbasis.

Let $\widehat{G}$ denote the set of equivalence classes of irreducible unitary representations of $G$, and for each $\pi\in\widehat{G}$ fix an orthogonal basis $\{e_i^\pi\}$ of the representation space $V_\pi$. The associated matrix coefficients
\[
\pi_{ij}(g) = \langle e_i^\pi,\pi(g)\,e_j^\pi\rangle, \qquad 1\leq i,j\leq \dim V_\pi,
\]
span a subspace $\mathcal{E}_\pi\subset C^\infty(G)$ of dimension $(\dim V_\pi)^2$, and the Peter--Weyl theorem provides the orthogonal decomposition
\[
L^2(G) =\overline{\bigoplus_{\pi\in\widehat{G}}\, \mathcal{E}_\pi}.
\]
\noindent Bi-invariance of the metric means that $\Delta$ commutes with both the left- and the right-regular representation of $G$. Schur's lemma thus forces $\Delta$ to act on every $\pi\in\widehat{G}$ by a non-negative scalar $\lambda_\pi$, and hence $\Delta\,\pi_{ij} = \lambda_\pi\,\pi_{ij}$ for all $i,j$. Distinct irreducibles may share the same eigenvalue, so the eigenspaces of $\Delta$ are the finite direct sums
\[
E_\lambda = \bigoplus_{\pi\,:\,\lambda_\pi = \lambda} \mathcal{E}_\pi,
\]
and every eigenfunction is a finite linear combination of matrix coefficients.

It remains to verify that $E_\lambda \cdot E_\mu$ is contained in a finite sum of eigenspaces. To this end, it suffices to consider the pointwise product of two matrix coefficients $\pi_{ij}$ and $\rho_{kl}$, with $\pi,\rho\in\widehat{G}$. Equipping $V_\pi\otimes V_\rho$ with the orthogonal basis $\{e_i^\pi\otimes e_k^\rho\}$, we have
\[
\pi_{ij}(g)\,\rho_{kl}(g) \;=\; (\pi\otimes\rho)_{(i,k),(j,l)}(g),
\]
so $\pi_{ij}\,\rho_{kl}$ is itself a matrix coefficient of the tensor product representation $\pi\otimes\rho$. This representation decomposes as a finite direct sum of irreducibles as
\[
\pi\otimes\rho \;\cong\; \bigoplus_{\sigma\in\widehat{G}} m_{\pi\rho}^{\sigma}\,\sigma, \qquad m_{\pi\rho}^{\sigma}\in\mathbb{Z}_{\ge 0},
\]
where the coefficients $m_{\pi\rho}^{\sigma}$ are zero for all but finitely many $\sigma$. Expanding $\pi_{ij}\,\rho_{kl}$ relative to a basis of matrix coefficients of the constituent irreducibles exhibits it as a finite sum of eigenfunctions of $\Delta$. We conclude that every eigenbasis of $G$ is a finite product eigenbasis.
\end{exa}

The same mechanism applies on any compact homogeneous space $G/H$ by restricting the analysis above to the subspace of $H$-invariant functions; see, for example, Lemma 9 in Section 2 of \cite{LinMendesRadeschi2025}. In the model case of the round sphere $\mathbb{S}^n = \mathrm{SO}(n+1)/\mathrm{SO}(n)$, the eigenspaces are the spaces of spherical harmonics, and the triple product constants of eigenfunctions are precisely the Clebsch--Gordan coefficients. Furthermore, a straightforward argument shows that a finite quotient of a space which admits a finite product eigenbasis itself admits such a basis. This shows that any manifold covered by a compact homogeneous space, such as the Klein bottle, admits a finite product eigenbasis.

\begin{exa}[Hyperbolic surfaces]
\label{ex: hyperbolic manifolds} Let $M=\Gamma\backslash \mathbb{H}$ be a closed arithmetic hyperbolic surface and choose an orthonormal Laplace eigenbasis $\{\phi_i\}$. Fix a nonconstant Laplace eigenfunction $\phi$ and define
\[
	c_i=\langle \phi^2,\phi_i\rangle.
\]
Then, by the work of \cite{BernsteinReznikov1999} and \cite{Reznikov2001}, there exist constants $c$ and $C$ such that
\[
	c\,\abs{\log\epsilon}\leq \sum c_i^2\,e^{(\frac{\pi}{2}-\epsilon)\sqrt{|4\lambda_i-1|}}\leq C\,\abs{\log\epsilon}^3
\]
as $\epsilon\to 0$. In particular, infinitely many of the coefficients $c_i$ are nonzero. This immediately implies that no eigenbasis of $M$ can be a finite product eigenbasis.

Recent work of Adve gives a complementary inverse perspective. The main theorem of \cite{Adve2025} can be viewed as a range characterization result, in terms of the conformal bootstrap equations, for the map which sends a hyperbolic surface to its spectrum together with its triple products.
\end{exa}

\begin{exa}[Abelian Cayley graphs]
\label{ex:cayley-ab}
    Like tori, Cayley graphs on abelian groups admit $2$-product eigenbases. Let $\Gamma$ be an abelian group and $S\subset\Gamma$ be a symmetric subset. The Laplacian matrix of $\Cay(\Gamma,S)$ has a complex eigenbasis consisting of
    \[e_\xi(x)=e^{2\pi i\langle\xi,x\rangle}\]
    for $\xi\in\widehat{\Gamma}=\operatorname{Hom}(\Gamma,\R/\Z)$ in the dual group of $\Gamma$. 
    Passing to the associated real basis of sines and cosines gives a $2$-product eigenbasis.
\end{exa}

\begin{exa}[Cayley sum graphs]
\label{ex:cayley-sum}
    For $\Gamma$ abelian and $S\subset\Gamma$ satisfying $S=-S$, the Cayley sum graphs $G:=\Cay^+(\Gamma,S)$ also admit $2$-product eigenbases. Letting $\Delta$ be the Laplacian matrix of $G$, we have
    \begin{align*}
        \Delta e_\xi
        &=\abs{S}e_\xi-\sum_{s\in S}e^{2\pi i\langle\xi,s\rangle}e_{-\xi}\\
        \Delta e_{-\xi}
        &=-\sum_{s\in S}e^{-2\pi i\langle\xi,s\rangle}e_{\xi}+\abs{S}e_{-\xi}.
    \end{align*}
    Since $S=-S$, the sum $\sum_{s\in S}e^{2\pi i\langle\xi,s\rangle}$ is real.
    These two identities imply that the elements of the real basis of sines and cosines associated to $\{e_\xi:\xi\in\widehat\Gamma\}$ is an eigenbasis of $G$.
    As in \cref{ex:cayley-ab}, this is a $2$-product eigenbasis.
\end{exa}

\begin{exa}[Nearly abelian Schreier graphs]
\label{ex:cayley-gen}
    Let $\Gamma$ be a finite group all of whose irreducible representations have dimension at most $d$. Let $S\subset\Gamma$ be a symmetric generating set and let $H$ be a subgroup.
    An argument analogous to that of \cref{ex: Lie groups} shows that the Schreier coset graph $\Sch(\Gamma,H,S)$ admits a $4d^3$-product eigenbasis. 

    First, we produce a $d^3$-product eigenbasis formed of complex vectors.
    Indeed, the Laplacian $\Delta$ of the Schreier coset graph $\Sch(\Gamma,H,S)$ can be written as
    \[\sum_{s\in S}(\rho_{H\backslash\Gamma}(1)-\rho_{H\backslash\Gamma}(s)),\]
    where $\rho_{H\backslash\Gamma}$ is the permutation representation of $\Gamma$ on the right cosets $H\backslash\Gamma$. 
    For each irreducible representation $\pi$ of $\Gamma$, the $\pi$-isotypic component of $\R^{H\backslash\Gamma}$ is preserved by $\Delta$. Take an eigenbasis of $\Delta$ which refines the isotypic decomposition of $\R^{H\backslash\Gamma}$. (Unlike in \cref{ex: Lie groups}, $\Delta$ need not act by a scalar on the isotypic components.) We may enforce that the eigenbasis we choose is preserved under complex conjugation.

    If $\varphi$ and $\psi$ are eigenvectors of $\Delta$ in isotypic components for the irreducible representations $\pi$ and $\rho$ of $\Gamma$, respectively, the vertexwise product $\varphi\psi$ is contained within the sum of the $\sigma$-isotypic components of $\R^{H\backslash\Gamma}$ for those $\sigma$ which appear in the decomposition of the tensor product $\pi\otimes\rho$. 
    Writing $\pi\otimes\rho\cong\bigoplus_{\sigma}m_{\pi\rho}^\sigma\sigma$,
    the dimension of the sum of these $\sigma$-isotypic components is at most
    \[\sum_{\sigma:m_{\pi\rho}^\sigma>0}(\dim\sigma)^2\leq d\sum_{\sigma:m_{\pi\rho}^\sigma>0}\dim\sigma\leq d\cdot\dim(\pi\otimes\rho)\leq d^3.\]
    We conclude that $\varphi\psi$ can be written as a linear combination of at most $d^3$ basis elements. Finally, it can be explicitly computed that replacing a complex conjugate pair $(\varphi,\overline\varphi)$ of eigenfunctions by $(\Re\varphi,\Im\varphi)$ transforms our $d^3$-product complex eigenbasis to a $4d^3$-product real eigenbasis. (See also the proof of \cref{lem:upper bound on N}.)

    We remark that a (qualitatively) equivalent characterization of such groups $\Gamma$ is given by results of Ito (see \cite[Proposition~1.1]{IsaacsPassman}) and Isaacs and Passman \cite[Theorem~D]{IsaacsPassman}. 
    If a finite group has an abelian normal subgroup of index $d$, then each of its irreducible representations have dimension at most $d$.
    Conversely, there is some function $f(d)$ such that, if all of the irreducible representations of $\Gamma$ have dimension at most $d$, then $\Gamma$ has an abelian normal subgroup of index at most $f(d)$.\footnote{Isaacs and Passman state this result for abelian subinvariant subgroups rather than abelian normal subgroups.
    However, any abelian subgroup $H\leq\Gamma$ of bounded index $N$ can be made into an abelian normal subgroup of index at most $N!$ by intersecting the conjugates $gHg^{-1}$ for $g$ ranging through a set of left coset representatives of $H$.
    So, qualitatively, all such conditions are equivalent.}
    This explains the conjectural description we gave in \cref{prob:N-prod-graph} of graphs admitting an $N$-product eigenbasis.
\end{exa}

\begin{exa}[Graph joins]\label{ex:graph-join}
    Let $G_1$ and $G_2$ be two graphs each admitting $N$-product eigenbases containing their respective all-ones vectors.
    Let $G$ be the graph join of $G_1$ and $G_2$, formed by connecting every vertex of $G_1$ to every vertex of $G_2$ by an edge. 

    Consider the basis of $\R^{V(G_1)\sqcup V(G_2)}$ consisting of the all-ones vector $\varphi_0$, the vector $\varphi_1$ which assigns $\abs{V(G_2)}$ to those vertices in $G_1$ and $-\abs{V(G_1)}$ to those vertices in $G_2$, and the vectors formed by taking nontrivial eigenvectors of $\Delta_{G_1}$ and $\Delta_{G_2}$ and padding them by zeros on the other graph. One can check that this is an orthogonal $\max(2,N+1)$-product eigenbasis of $\R^{V(G_1)\sqcup V(G_2)}$.

    Some examples of graphs admitting $3$-product bases that can be obtained in this manner are wheel graphs (formed by adding a universal vertex to a cycle) and complete graphs with one edge removed.
    Using that Cayley graphs over $(\Z/2\Z)^d$ admit $1$-product bases, this construction also produces some examples of graphs admitting $2$-product bases which do not fall into the classification given by \cref{thm: 2PP iff Cayley}, such as the complete bipartite graph $K_{n_1,n_2}$ with $n_i=2^{k_i}$ and the complete graph $K_n$ with one edge removed for $n=2^k+2$.
\end{exa}

\vspace{3em}

\section{Deducing Flatness From 2-Product Bases}
\label{sec: 2PP implies flat}

We now present the proof of \cref{prop: 2 product eigenbasis implies flatness}. To begin, we study the structure of a single eigenfunction within a $2$-product eigenbasis. In \cref{lem: gradient norm}, we record the useful fact that the gradient of the elements of such a basis have a rigid structure. We apply this in Lemmas \ref{lem: level sets are minimal} and \ref{lem: tilde phi} to construct a companion eigenfunction on the universal cover of $M$, from which we extract a phase. Then, we show in \cref{lem: angles are harmonic}  that we can generate infinitely many eigenfunctions along this phase, and in \cref{lem: constant angle inner product} that therefore the phase functions must form a local basis exhibiting flatness. The proof of the proposition follows.

Let us therefore consider a manifold $(M,g)$ which we suppose to admit a $2$-product eigenbasis. For simplicity, suppose $M$ has volume $1$. Fix an eigenfunction $\phi$ with eigenvalue $\lambda>0$ in the distinguished basis. Without loss of generality, we assume $\phi$ is scaled so as to have $L^\infty$-norm equal to $1$.

\begin{lemma}\label{lem: gradient norm}
    The eigenfunction $\phi$ satisfies
    \begin{equation}
    \label{eq: norm phi squared}
        |\grad \phi|^2=\lambda (1-\phi^2).
    \end{equation}
\end{lemma}

\begin{proof}
    The function $\phi^2$ is everywhere nonnegative and a sum of at most two basis elements, so we can write $\phi^2=c+\psi$ for a constant $c=\norm{\phi}_2^2$ and an eigenfunction $\psi$ with eigenvalue $\mu$. Compute
    \begin{equation}
    \label{eq: phi2 decom}
        \mu(\phi^2-c)=\mu \psi=\Delta(c+\psi)=\Delta \phi^2=2\phi\Delta \phi-2\abs{\grad \phi}^2=2\lambda \phi^2-2\abs{\grad \phi}^2.
    \end{equation}
    Evaluating at a local extremum $p$ of $\phi$, where $\grad\phi(p)=0$, we obtain
    \begin{equation}
    \label{eq: phi2 max}
        0=(2\lambda-\mu)\phi(p)^2+c\mu.
    \end{equation}
    In particular, this holds at any global maximum or minimum of $\phi$, so the range of $\phi$ is exactly $[-1,1]$. The function $\psi=\phi^2-c$ is also an eigenfunction, with range $[-c,1-c]$; since $\psi$ is a scaled basis element, its range must be symmetric about $0$ by the same reasoning. This forces $c=1/2$. Applying \eqref{eq: phi2 max} at a global maximum of $\phi$ now gives
    \[0=(2\lambda-\mu)+\tfrac12\mu,\]
    so $\mu=4\lambda$. Substituting $\mu=4\lambda$ and $c=1/2$ into \eqref{eq: phi2 decom} concludes the proof.
\end{proof}

The relation \eqref{eq: norm phi squared} implies that $\phi$ is \emph{transnormal}, meaning that the norm of its gradient is a function of its value. Combined with $\Delta \phi=\lambda\phi$, this makes $\phi$ \emph{isoparametric}, meaning that it is transnormal and, furthermore, its Laplacian is a function of its value. The level sets of such functions, originally studied by Cartan in \cite{Cartan1939}, exhibit highly regular geometry. In our case, we obtain the following consequence.

\begin{lemma}
\label{lem: level sets are minimal}
	The level sets of $\phi$ are minimal smooth hypersurfaces of codimension $1$.
\end{lemma}

\begin{proof}
	We first establish the claim away from the critical set $\{\grad \phi=0\}$. There, set $\nu =\grad \phi/\abs{\grad \phi}$ and compute
	\begin{equation}
	\label{eq: div}
		\textrm{div}\, \nu =\frac{-\lambda \phi}{\abs{\grad \phi}}-\Big\langle\grad \phi, \frac{\grad\abs{\grad \phi}}{\abs{\grad \phi}^2}\Big\rangle.
	\end{equation}
	By \cref{lem: gradient norm},
    \[
        \grad\b{|\grad \phi|^2}=\lambda\grad \b{1-\phi^2}=-2\lambda \phi\,\grad \phi\quad \implies\quad 
        \grad\abs{\grad \phi}=\frac{-\lambda \phi\,\grad \phi}{\abs{\grad \phi}},
    \]
    from which $\textrm{div}\,\nu=0$ follows after plugging back in to \eqref{eq: div}. Note that this is precisely the condition that the regular level sets are minimal. A simple application of Stokes' theorem then yields that the regular level sets have constant $(n-1)$-volume.

    To handle the critical set, we apply a result of Wang characterizing the level sets of transnormal functions. By \cite[Theorem A]{Wang1987}, the critical level sets of $\phi$ are smooth and the regular level sets are \emph{tubes} over the critical level sets, that is, images of spheres under the exponential map in the normal direction away from the critical level set. Since the regular level sets have constant volume, these tubes must reduce to two exponentiated copies of the critical level set, and the critical level set must therefore have codimension $1$, completing the proof.
\end{proof}

\cref{lem: level sets are minimal} provides a rigid geometric description of the critical level sets, which we now leverage to construct from $\phi$ a new eigenfunction with the same eigenvalue. Denote by $\widetilde M$ the universal cover of $M$, equipped with the natural pullback metric, and by $\Phi$ the lift of $\phi$ to $\widetilde M$.

\begin{lemma}
\label{lem: tilde phi}
	There exists a Laplace eigenfunction $\Psi$ on $\widetilde M$ with eigenvalue $\lambda$ such that $\Phi^2+\Psi^2=1$.
\end{lemma}

\begin{proof}
	Away from the critical set $\{\phi=\pm 1\}=\{\grad \phi =0\}$, the function $1-\phi^2$ is smooth and positive, so its square root admits two smooth branches $\pm\sqrt{1-\phi^2}$. We view these branches as sections of the trivial bundle $M\times \R$, and describe how to glue them together across the critical set.

    By \cref{lem: level sets are minimal}, the critical set is a smooth codimension $1$ hypersurface. To compute the Hessian $H(p)$ of $\phi$ at a critical point $p$, fix normal coordinates $(x_1,\dots, x_n)$ about $p$. Differentiate \eqref{eq: norm phi squared} to obtain
    \[
        2\sum_{k}(\partial_k\phi)(\partial_{ik}\phi)=\partial_i|\grad \phi|^2=\lambda\partial_i(1-\phi^2)=-2\lambda \phi (\partial_i\phi).
    \]
    Differentiating once more yields
    \[
        \sum_k (\partial_{jk}\phi)(\partial_{ik}\phi)+(\partial_k\phi)(\partial_{ijk}\phi)=-\lambda\big((\partial_i\phi)(\partial_j\phi)+\phi\,\partial_{ij}\phi\big),
2    \]
    which we evaluate at $p$, where $\grad \phi(p)=0$, to deduce
    \[
        \sum_k (\partial_{jk}\phi)(\partial_{ik}\phi)=-\lambda \phi(p)(\partial_{ij}\phi),
    \]
    so $H(p)^2=-\lambda \phi(p)H(p)$. Assuming $\phi(p)=1$, each eigenvalue of $H(p)$ is therefore either $0$ or $-\lambda$, and since $\textrm{tr}(H(p))=-\Delta\phi(p)=-\lambda$, exactly one eigenvalue is $-\lambda$. The function $\phi$ is thus Morse--Bott near $p$.

    By the Morse--Bott Lemma, we may choose local coordinates $(y_1,\dots, y_n)$ near $p$ in which $\phi(y)=1-y_n^2$. Then
	\[
        \sqrt{1-\phi^2}=\sqrt{1-(1-y_n^2)^2}=y_n\sqrt{2-y_n^2},
    \]
	which is smooth near $y_n=0$, so we smoothly may attach the two branches of $\sqrt{1-\phi^2}$ across the critical set by tracking the sign of $y_n$. The case $\phi(p)=-1$ is analogous.

    We therefore obtain a double cover of $M$ on which we may smoothly choose a sign for the square root. By passing to the universal cover and picking such a sign, we obtain a function $\Psi$ on $\widetilde M$, and, by construction, $\Phi^2+\Psi^2=1$.

    It remains to verify that $\Psi$ is an eigenfunction with eigenvalue $\lambda$. Applying \cref{lem: gradient norm} once more,
    \[
        \Delta \sqrt{1-\phi^2}=\frac{|\grad \phi|^2}{(1-\phi^2)^{3/2}}-\frac{\lambda \phi^2}{\sqrt{1-\phi^2}}=\lambda \sqrt{1-\phi^2}
    \]
    away from the critical set of $\phi$, and the result follows after extending by continuity to the critical set.
\end{proof}

We may therefore write $\Phi=\cos(\alpha)$ with a phase function $\alpha=(\Phi,\Psi)\colon \widetilde M\to \mathbb{S}^1\subset \R^2$.

\begin{lemma}
\label{lem: angles are harmonic}
    The function $\alpha$ satisfies $|\grad \alpha|^2=\lambda$ and $\Delta\alpha =0$. Furthermore, for every $k\geq 0$, the function $\cos (k\alpha)$ is the lift of an element of the fixed $2$-product eigenbasis on $M$.
\end{lemma}

\begin{proof}
    Compute
    \[
        |\nabla\alpha|^2\,\sin^2(\alpha)=|\nabla \Phi|^2=\lambda(1-\Phi^2)=\lambda \sin^2(\alpha),
    \]
    so $|\nabla\alpha|^2=\lambda$ wherever $\sin(\alpha)\neq 0$, and the identity extends by continuity. Then
    \[
        \lambda \cos(\alpha)=\Delta\cos(\alpha)=|\grad\alpha|^2\cos(\alpha)-\sin(\alpha)\,\Delta\alpha=\lambda \cos(\alpha)-\sin(\alpha)\,\Delta\alpha,
    \]
    forcing $\Delta \alpha=0$. It then follows that $\Delta\cos(k\alpha)=k^2\lambda \cos(k\alpha)$ for every $k\geq 0$.
    
    We prove the second claim by induction. When $k=0$, the claim is immediate. When $k=1$, we have $\cos(k\alpha)=\cos(\alpha)=\Phi$, which is the lift of $\phi$ to $\widetilde M$. Then, by the identity
    \[
        2\cos(\alpha) \cos( k\alpha) = \cos((k+1)\alpha)+\cos((k-1)\alpha)
    \]
    together with the fact that the eigenfunctions lie in a $2$-product eigenbasis, the result follows.
\end{proof}

We must now make use of our $2$-product eigenbasis, beyond squaring a single element, in order to deduce information about the metric. To this end, write the lift of two non-constant elements of the distinguished basis as $\cos(\alpha)$, with eigenvalue $\lambda$, and $\cos(\beta)$, with eigenvalue $\mu$, on $\widetilde M$. We then have:

\begin{lemma}
\label{lem: constant angle inner product}
    The inner product $\langle \grad \alpha, \grad \beta\rangle$ is constant on $\widetilde M$.
\end{lemma}

\begin{proof}
    Set
    \[
        u_k\coloneqq \cos(k\alpha)\cos(k\beta).
    \]
    By \cref{lem: angles are harmonic}, $\cos(k\alpha)$ and $\cos(k\beta)$ are lifts of elements of our $2$-product eigenbasis, so there exist coefficients $a_k,b_k, c_k\in\R$, not all zero, such that
    \[
        \big(a_k(k^{-2}\Delta)^2+b_k(k^{-2}\Delta)+c_k\big)u_k=0.
    \]
    Rescale these coefficients so that $\max(|a_k|,|b_k|,|c_k|)=1$, and pass to a subsequence to obtain polynomials
    \[
        P_k(t)=a_kt^2+b_kt+c_k\to P(t)=at^2+bt+c\neq 0
    \]
    for which
    \[
        P_k(k^{-2}\Delta)u_k=0.
    \]

    Now, for any harmonic function $\theta\in C^\infty(U)$ defined on an open set $U\subset \widetilde M$, write $q=|\grad\theta|^2$ and compute that
    \[
        (k^{-2}\Delta)(e^{ik\theta})=qe^{ik\theta},\qquad (k^{-2}\Delta)^2(e^{ik\theta})=\big(q^2-2ik^{-1}\langle \nabla q,\nabla\theta\rangle+k^{-2}\Delta q\big)e^{ik\theta},
    \]
    and hence
    \[
        e^{-ik\theta}P_k(k^{-2}\Delta)(e^{ik\theta})=Q_k(\theta)\coloneqq a_k\big(q^2-2ik^{-1}\langle \nabla q,\nabla\theta\rangle+k^{-2}\Delta q\big)+b_kq+c_k.
    \]
    Note that $Q_k(\theta)\to P(q)$ in $C^\infty_\textrm{loc}(U)$ as $k\to\infty$.
    
    We apply this reasoning to $u_k$. Expand $u_k$ as
    \[
        u_k=\frac{1}{4}\big(e^{ik\theta_+}+e^{-ik\theta_+}+e^{ik\theta_-}+e^{-ik\theta_-}\big),
    \]
    where $\theta_\pm=\alpha\pm\beta$. By \cref{lem: angles are harmonic}, both $\theta_\pm$ are harmonic, and $\theta_+\pm\theta_-$ is either $2\alpha$ or $2\beta$, both of which have non-vanishing gradient.  Set $q_\pm=\abs{\grad \theta_\pm}^2$, fix the open set $U=\{q_+\neq 0\}\subset \widetilde M$, choose a bump function $\chi\in C_c^\infty(U)$, and write
    \begin{align*}
        0&=\int_{\widetilde M} \chi e^{-ik\theta_+}P_k(k^{-2}\Delta)u_k\\
        &=\frac{1}{4}\sum_{\sigma\in \{\pm\theta_\pm\}}\int_{\widetilde M} \chi e^{-ik\theta_+}P_k(k^{-2}\Delta)(e^{ik\sigma})\\
        &=\frac{1}{4}\sum_{\sigma\in \{\pm\theta_\pm\}}\int_{\widetilde M} \chi e^{ik(\sigma-\theta_+)}Q_k(\sigma).
    \end{align*}
    Take $k\to\infty$. The term with $\sigma=\theta_+$ converges to $\int\chi P(q_+)$. In each of the remaining terms, the gradient of the phase $\sigma-\theta_+\in\{-2\theta_+, -2\alpha, -2\beta \}$ does not vanish, while $Q_k(\sigma)$ converges smoothly on compact sets to $P(|\nabla\sigma|^2)$, so these terms vanish in the limit by the principle of nonstationary phase. Since $\chi$ was arbitrary, $P(q_+)=0$ pointwise on $U$, and so $q_+$ is constant on $U$ since it is continuous and either zero or the root of a quadratic. Continuity then extends this constancy to all of $\widetilde M$. The result now follows from
    \[
        q_+=|\grad\alpha|^2+|\grad\beta|^2+2\langle\grad\alpha, \grad\beta\rangle=\lambda+\mu+2\langle\grad\alpha, \grad\beta\rangle,
    \]
    where we have once again applied \cref{lem: angles are harmonic}.
\end{proof}

The phase functions therefore provide a means to deduce flatness. We exploit this to complete the proof.

\begin{proof}[Proof of \cref{prop: 2 product eigenbasis implies flatness}]
    By the embedding theorem of \cite{BerardBessonGallot1994}, sufficiently many elements of any orthogonal eigenbasis of $L^2(M)$ embed $M$ into Euclidean space. Fixing $p\in M$, we may therefore choose elements $\phi_1, \dots, \phi_k$ of the distinguished basis such that $\{\grad \phi_i(p)\}_{i=1}^k$ is a basis of $T_pM$. Lifting to $\widetilde M$ and choosing a lift $\tilde p$ of $p$, we obtain corresponding phase functions $\alpha_1,\dots, \alpha_k$ such that $\{\grad \alpha_i(\tilde p)\}_{i=1}^k$ forms a basis of $T_{\tilde p}\widetilde M$. By \cref{lem: constant angle inner product}, the metric coefficients $\langle\grad \alpha_i, \grad \alpha_j\rangle$ are constant in a neighborhood of $\tilde p$, so the metric is flat near $p$. Since $p$ was arbitrary, $\widetilde{M}$ is flat, and hence $M$ is flat.
\end{proof}

\vspace{3em}

\section{\texorpdfstring{$N$}{N}-Product Bases on Flat Manifolds}
\label{sec: flat manifolds and the NPP}

We turn our attention to the study of $N$-product bases on flat manifolds. We begin by recalling some useful background, following \cite{Charlap1986}, which provides a more general introduction. We then prove \cref{prop: n product bases and holonomy} and immediately apply it to complete the proof of \cref{thm: 2 product eigenbasis iff flat torus}.

Every flat manifold is of the form $M\cong \R^n/\Gamma$, where $\Gamma$ is a Bieberbach group \cite[Chapter~II,~Corollary~5.1]{Charlap1986}. 
By Bieberbach's first theorem \cite[Chapter~I,~Theorem~3.1]{Charlap1986}, every Bieberbach group fits into a short exact sequence
\begin{equation}
\label{eq: Bieberbach spectral sequence}
    0\to \Lambda\to \Gamma\to H\to 0,
\end{equation}
where $\Lambda$ is a full-rank lattice of translations in $\R^n$ and where $H$, the \textit{holonomy group} of $M$, is identified with a finite subgroup of $\Orth_n(\R)$. From \eqref{eq: Bieberbach spectral sequence}, we see that the torus $T\coloneqq \R^n/\Lambda$ is a finite Riemannian cover of $M$ with deck group $H$. For each $h\in H$, choose a lift $A_hx+b_h$ to the affine group over $\R^n$, where $A_h\in O_n(\R)$ denotes the linear part and $b_h\in \R^n$ denotes the shift. By abuse of notation, we also write $h$ for the affine map $x\mapsto A_hx+b_h$. Note that no element can have $A_h=-\operatorname{Id}$, since $\Gamma$ must be torsion-free.

The space $L^2(M)$ is identified with the $H$-invariant subspace of $L^2(T)$, which we denote by $L^2(T)^H$. In particular, the lift to $T$ of any eigenfunction of $M$ is an $H$-invariant eigenfunction of $T$. Letting $\xi$ range over the dual lattice $\Lambda^*$ and writing 
\[
    e_\xi(x)=e^{2\pi i\langle\xi,x\rangle}
\]
as in \cref{ex: torus}, we may decompose any such lift in the Fourier basis $\{e_\xi\}_{\xi\in\Lambda^*}$ of $L^2(T)$.

To establish \cref{prop: n product bases and holonomy}, we split the claim into the lower and the upper bound. Recall that $\nu(M)\coloneqq\inf\{N : \textup{$M$ admits an $N$-product eigenbasis}\}$.

\begin{lemma}\label{lem:upper bound on N}
    A compact flat manifold $M$ with holonomy group $H$ admits a $4\abs{H}$-product eigenbasis.
\end{lemma}
\begin{proof}
We naturally have a projection
\[
    P\colon L^2(T;\C)\longrightarrow L^2(T;\C)^H\cong L^2(M;\C),
        \qquad
        Pf=\frac{1}{\abs{H}}\sum_{h\in H}f\circ h,
\]
which we employ to define $u_\xi\coloneqq Pe_\xi$. Since $u_{A_h^*\xi}$ is a multiple of $u_\xi$ by a complex number with unit length, these functions depend up to scaling only on the $H$-orbit of $\xi$. A standard check reveals that choosing one $\xi$ from each orbit, discarding zeros, and normalizing, gives a complex eigenbasis of $L^2(M;\C)$. We then compute
\[
    u_\xi u_\eta=P(e_\xi)u_\eta=P(e_\xi u_\eta)=\frac{1}{\abs{H}}\sum_{h\in H}e^{2\pi i\ip{\eta, b_h}} P(e_{\xi+A_h^*\eta})=\frac{1}{\abs{H}}\sum_{h\in H}e^{2\pi i\ip{\eta, b_h}} u_{\xi+A_h^*\eta},
\]
which expresses the product of two basis elements as the sum of at most $\abs{H}$ basis elements.

We convert this observation into the desired conclusion by noting that, since $\overline{u_\xi}=u_{-\xi}$, this basis is stable under conjugation, so taking real and imaginary parts gives a basis over $\R$. A product of two real basis elements expands into four complex products, hence into at most $4\abs{H}$ elements of the complex basis. As the product is real, its complex support is invariant under conjugation, so passing back to the real basis costs no further terms and the result follows.
\end{proof}

The lower bound on $\nu(M)$ is slightly more technical.

\begin{lemma}
\label{lem:lower bound on N}
    If a compact flat manifold with holonomy group $H$ admits an $N$-product eigenbasis, then $2\abs{H}\leq N$.
\end{lemma}
\begin{proof}
    Fix an $N$-product eigenbasis throughout. Let $H_1$ be the group generated by $H$ and $-\operatorname{Id}$, so that $\abs{H_1}=2\abs{H}$.
    For $h\in H_1$, let $h^*$ be the dual action on $\Lambda^*$.
    
    For distinct $h_1,h_2\in H_1$, the equation $(h_1^*-h_2^*)\xi=0$ cuts out a proper subspace of the dual vector space $(\R^n)^*=\Lambda^*\otimes_\Z\R$. 
    Since $H_1$ is finite and $\Lambda^*$ is cocompact in $(\R^n)^*$, we can find some $\xi_0\in\Lambda^*$ whose $H_1$-orbit has full size $\abs{H_1}$.
    The function $\text{Re}(u_{\xi_0})$ is an eigenfunction on $M$ whose Fourier coefficient at $\xi_0$ is nonzero. Find some basis element $\varphi$ whose Fourier coefficient at $\xi_0$ is also nonzero, and write it as
    \[\varphi=\sum_{\xi\in\Xi}a_\xi e_\xi\]
    for some finite set $\Xi\ni\xi_0$ and some nonzero $a_\xi\in\C$.
    Since $\varphi$ is to be an eigenfunction, we must have $\norm{\xi}=\norm{\xi_0}$ for every $\xi\in\Xi$; that is, the Fourier support of $\varphi$ lies entirely in the sphere of radius $\norm{\xi_0}$.

    Now that we have fixed $\varphi$ and $\xi_0$, choose some $\eta_0\in\Lambda^*$ subject to the following constraints:
    \begin{enumerate}
        \item For each choice of $h_1,h_2\in H_1$ distinct, we have $\langle(h_1^*-h_2^*)\xi_0,\eta_0\rangle\neq 0$.

        \item For each $\xi\in\Xi\setminus\{\xi_0\}$ and each $h\in H_1$, we have $\langle\xi-\xi_0,\xi_0+h^*\eta_0\rangle\neq0$.
    \end{enumerate}
    Those $\eta_0$ which fail to satisfy (1) or (2) lie in finitely many hyperplanes, and so $\eta_0$ can be chosen appropriately.
    Similarly to our choice of $\varphi$, let
    \[\psi=\sum_{\eta\in\Sigma}c_\eta e_\eta\]
    be a basis element whose coefficient $c_{\eta_0}$ of $e_{\eta_0}$ is nonzero. Note that every $\eta\in\Sigma$ satisfies $\norm{\eta}=\norm{\eta_0}$.

    Our main claim is the following: if $(\xi,\eta,h)\in\Xi\times\Sigma\times H_1$ is such that $\xi+\eta=\xi_0+h^*\eta_0$, then $\xi=\xi_0$. Indeed, under this condition we have
    \[\langle\xi_0,h^*\eta_0\rangle=\frac12\left(\norm{\xi_0+h^*\eta_0}^2-\norm{\xi_0}^2-\norm{h^*\eta_0}^2\right)=\frac12\left(\norm{\xi+\eta}^2-\norm{\xi}^2-\norm{\eta}^2\right)=\langle\xi,\eta\rangle,\]
    and so
    \[\langle\xi-\xi_0,\xi_0+h^*\eta_0\rangle=\langle\xi,\xi+\eta\rangle-\langle\xi_0,\xi_0+h^*\eta_0\rangle=\norm{\xi}^2+\langle\xi,\eta\rangle-\norm{\xi_0}^2-\langle\xi_0,h^*\eta_0\rangle=0.\]
    This contradicts condition (2) of our choice of $\eta_0$ when $\xi\neq \xi_0$.

    Now, consider the product $\varphi\psi$. For each $\kappa\in\Lambda^*$, we have
    \begin{equation}\label{eq:fourier-product}
    \langle\varphi\psi,e_\kappa\rangle=\sum_{\substack{\xi\in\Xi,\ \eta\in\Sigma\\\xi+\eta=\kappa}}a_\xi c_\eta.
    \end{equation}
    By our main claim, the only contribution to this sum when $\kappa=\xi_0+h^*\eta_0$ is $a_{\xi_0}c_{h^*\eta_0}$. Since $\psi$ descends to a real function on $M$ and is $H$-invariant, we have
    \[c_{h^*\eta}=e^{2\pi i \ip{\eta, b_h}}c_\eta\quad\text{for }h\in H\qquad \text{and}\qquad \overline{c_{\eta}}=c_{-\eta}.\]
    In particular, $c_{h^*\eta_0}\neq 0$ for each $h\in H_1$.   
    This implies that the inner product in \eqref{eq:fourier-product} is nonzero for $\kappa=\xi_0+h^*\eta_0$ for each $h\in H_1$. In particular, for each $h\in H_1$, the decomposition of $\varphi\psi$ into our eigenbasis contains some eigenfunction with eigenvalue $4\pi^2\norm{\xi_0+h^*\eta_0}^2$.

    What remains to show is that, as $h\in H_1$ varies, $4\pi^2\norm{\xi_0+h^*\eta_0}^2$ takes on $\abs{H_1}$ different values. Equivalently, we must show that $\langle\xi_0,h^*\eta_0\rangle$ takes $\abs{H_1}$ different values. Indeed, using the fact that $H_1\leq\Orth_n(\R)$, this is exactly assumption (1) on $\eta_0$. Expressing $\varphi\psi$ in the basis therefore requires at least $\abs{H_1}=2\abs{H}$ terms, and hence $2\abs{H}\leq N$.
\end{proof}

Together, \cref{lem:upper bound on N,lem:lower bound on N} imply \cref{prop: n product bases and holonomy} directly.

We may now finally complete the proof of \cref{thm: 2 product eigenbasis iff flat torus}. In the previous section, we showed \cref{prop: 2 product eigenbasis implies flatness}; i.e., that if $(M,g)$ admits a $2$-product eigenbasis then it is flat.
\cref{lem:lower bound on N} then implies that the holonomy group $H$ of $M$ satisfies $2\abs{H}\leq \nu(M) \leq 2$. So $H$ must be trivial, hence $M$ is a torus, as claimed. The converse was established in \cref{ex: torus}.

\begin{exa}
\label{exa: twisted torus}
    The lower bound given in \cref{lem:lower bound on N} need not be sharp.
    Let $\Lambda\subset\R^3=\text{span}(e_1,e_2,e_3)$ be the direct sum of the hexagonal lattice in $\R^2=\text{span}(e_1,e_2)$ and $\Z e_3$. Set $\rho$ to be the rotation by $2\pi/3$ around the $e_3$ axis, which preserves the hexagonal lattice, and take $\Gamma$ to be generated by $\Lambda$ and $\rho x+\tfrac13e_3$. This is a Bieberbach group, and $M=\R^3/\Gamma$ is a closed flat $3$-manifold with holonomy $H\cong\Z/3$.
    
    We claim that $M$ admits no $6$-product eigenbasis, hence $7\leq \nu(M)\leq 12$, and outline the proof here. Suppose for contradiction that it did. Repeating the argument of \cref{lem:lower bound on N} shows that every basis element has Fourier support in a single $H_1$ orbit $\{\pm \xi, \pm \rho \xi, \pm\rho^2\xi\}$, for some $\xi\in \Lambda^*$.
    Indeed, if the Fourier support of a basis element met two such orbits, we could apply the same argument as in the lemma, tracking both orbits, to contradict the $6$-product assumption. We conclude that each basis element can meet at most one orbit.

    When $\xi$ is not on the $e_3$ axis, the real functions with Fourier support in $H_1\cdot \xi$ are precisely $\{\textup{Re}(au_\xi)\st a\in \C\}$. It now follows that our basis must intersect this space at exactly two elements, thus singling out two orthogonal lines in this space. Denote the angle these two lines form relative to $\textup{Re}(u_\xi)$ by $c(\xi)\in \R/(\pi/2)\Z$.

    Next, observe that $c(\rho\xi)=c(\xi)-2\pi\ip{\xi,e_3}/3\mod \pi/2$, simply since $u_{\rho\xi}$ is obtained by rotating $u_\xi$. Furthermore, after expanding $u_\xi u_\eta$ and replicating once again an argument similar to \cref{lem:lower bound on N}, one deduces that $c(\xi+\eta)=c(\xi)+c(\eta)$. Indeed, for generic $\xi$ and $\eta$, the six nonzero orbit components of the product lie in distinct eigenspaces, so the $6$-product assumption forces the $\xi+\eta$ component to lie on a line of phase $c(\xi)+c(\eta)$.
    We are now prepared to obtain our contradiction.

    Fix $0\neq\eta\in\Lambda^*$ horizontal, and define $\kappa:=\rho\eta - \eta$ and $\xi_k=\eta+ke_3$ for $k\in\{0,1\}$. Observe $\rho\xi_k=\xi_k+\kappa$ and apply the above properties of $c$ to conclude
    \[
        c(\kappa)=c(\rho\xi_k)-c(\xi_k)=-2\pi k/3\mod \pi/2.
    \]
    However, this cannot hold for both $k=0$ and $k=1$, and the claim follows.
\end{exa}

\vspace{3em}

\section{2-product bases in graphs}\label{sec:graph}

In this section, we prove \cref{thm: 2PP iff Cayley}; we will delay some computations to the following section.
Some portions of this proof will also apply to the manifold setting; we will use this to provide a second proof of \cref{prop: 2 product eigenbasis implies flatness}.

\subsection{An auxiliary graph}\label{sec:graph-aux}

Let $\mathcal A$ be an $\R$-algebra equipped with an inner product $\langle\cdot,\cdot\rangle$ for which pointwise multiplication by any element of $\mathcal A$ is self-adjoint, and let $\cB\subset\mathcal A$ be an orthogonal set of at most countable cardinality.
The case of interest to us will be when $\cB$ is an eigenbasis of the Laplacian either on a manifold $M$ (in which case $\mathcal A=C^\infty(M)$) or on a finite graph $G$ (in which case $\mathcal A\cong\R^{V(G)}$).
Suppose that the product of any two elements of $\cB$ can be written as a linear combination of at most $N$ elements of $\cB$. 
Pick any $\varphi\in\cB$. Construct a labeled graph $H_\varphi$ with vertex set $\cB$ and in which $\psi,\psi'\in\cB$ are connected by an edge of label $\langle4\varphi\psi,\psi'\rangle$ if this inner product is nonzero. 
(Vertices may have self-loops, which contribute $1$ to their degree.)
This graph has maximum degree at most $N$. 

We will make essential use of the fact that the structure of graphs with maximum degree (at most) $2$ cannot be that complicated. 

\begin{obs}\label{obs:max-deg-2}
    Every connected component of a graph with self-loops (but no multi-edges) and maximum degree at most $2$ (on a finite or infinite number of vertices) must be of one of the following forms:
    \begin{enumerate}
        \item An isolated vertex.
        \item A vertex with a single self-loop.
        \item A finite path with no self-loops.
        \item A finite path with a self-loop at one end.
        \item A finite path with a self-loop at each end.
        
        \item A finite cycle.
        \item A path which is infinite in one direction and has no self-loop at the finite end.
        \item A path which is infinite in one direction and has a self-loop at the finite end.
        \item A bi-infinite path.
    \end{enumerate}
\end{obs}

\noindent By contrast, there is no similar characterization of the connected components of a graph of maximum degree at most $3$, say.
This is why we are unable to prove an analogue of \cref{thm: 2PP iff Cayley} for $N$-product bases for any $N>2$.

Now, suppose that the multiplicative identity $\1\in\mathcal A$ lies in $\cB$, and choose $\varphi\in\cB$ distinct from $\1$.
We distinguish the connected component of $H_\varphi$ containing $\1$ as a vertex, which we term the \emph{$\1$-component}. Since $\1\cdot\varphi=\varphi\in\cB$, the vertex $\1$ has no self-loop and has degree $1$ in $H_\varphi$, connected by an edge to $\varphi$. This means that the $\1$-component is of type (3), (4), or (7) in the above.
In our two settings of interest, our starting point will be to understand the $\1$-component.

We now specialize further.
If $\mathcal A=C^\infty(M)$ for some Riemannian manifold $(M,g)$ of unit volume and $\cB$ is a $2$-product eigenbasis for the Laplacian on $M$, we say that $(\mathcal A,\cB)$ is a \emph{$2$-product manifold datum}. If $\mathcal A=\R^{V(G)}$ for some connected graph $G$ with an odd number of vertices\footnote{This assumption is presumably not essential. However, it substantially simplifies many of the arguments and much of the casework in the following two sections, and so we make use of it whenever is convenient.} and $\cB$ is a $2$-product eigenbasis for the Laplacian on $G$, we say that $(\mathcal A,\cB)$ is a \emph{$2$-product odd graph datum}.
In both cases, we normalize the elements of $\cB$ so that $\1\in\cB$ and so that each other element of $\cB$ has $L^2$ norm $1/\sqrt2$.
There are some additional choices to be made; namely, any element $\psi\in\cB$ can be replaced by its negation. For $\varphi\neq\psi$, this negates all the labels of edges of $H_\varphi$ incident to $\psi$. (Negating $\varphi$ negates the labels of edges of $H_\varphi$ not incident to $\varphi$.) We allow ourselves to negate basis elements whenever it is convenient, and generally consider each basis element to be identified only up to sign.

In both settings, it will be useful to understand the values that eigenfunctions may take. 
An important parameter will be the image size $\abs{\im f}$. (This is either a positive integer or $\infty$.) 
Given a graph $G$ and a function $f\colon V(G)\to\R$, the \emph{distribution} $\cD(f)$ of $f$ is the probability distribution of $f(v)$ where $v$ is a uniformly random vertex of $G$. 
Given a compact unit-volume Riemannian manifold $(M,g)$ and a function $f\colon M\to\R$, its \emph{distribution} $\cD(f)$ is the pushforward by $f$ of the probability measure on $M$ arising from the volume form. 
We define some probability distributions of interest: for a positive integer $n$, let $\cD_n^{\mathrm{even}}$ (resp.\ $\cD_n^{\mathrm{odd}}$) be the distribution of $\cos(\pi 2t/n)$ (resp.\ $\cos(\pi(2t+1)/n)$) for $t$ uniform in $\Z/n\Z$.
Let $\cD^{\mathrm{cont}}$ be the distribution of $\cos(2\pi x)$ where $x\sim\operatorname{Unif}(\R/\Z)$.

\subsection{Proof outline}\label{sec:graph-outline}

We now outline the proof of \cref{thm: 2PP iff Cayley} (and the corresponding second proof of \cref{prop: 2 product eigenbasis implies flatness}), stating the main lemmas and deferring their proofs to the following section.

Select any $\varphi\in\cB\setminus\{\1\}$. Our main aim will be to understand the structure of the graph $H_\varphi$; we will use this understanding to obtain general information about the elements of $\cB$.
First, we understand the $\1$-component and extract information about $\varphi$.
Let $T_n$ denote the degree-$n$ Chebyshev polynomial of the first kind, so that $T_n(\cos x)=\cos(nx)$ for each $x$.

\begin{lemma}\label{lem:1-comp}
    Suppose $(\mathcal A,\cB)$ is a $2$-product manifold datum or a $2$-product odd graph datum.
    Select $\varphi\in\cB\setminus\{\1\}$ and suppose that $\abs{\im\varphi}\geq 7$.
    
    In the manifold case, the $\1$-component of $H_\varphi$ is a uni-infinite path consisting (up to sign) of $T_m(\varphi)$ for each $m$. Moreover, $\cD(\varphi)=\cD^{\mathrm{cont}}$.

    In the graph case, the $\1$-component of $H_\varphi$ is a finite path, possibly with a self-loop at one end. 
    Up to sign, the vertices in the path are $T_0(\varphi),T_1(\varphi),\ldots,T_{k-1}(\varphi)$, where $k=\abs{\im\varphi}$. 
    Moreover, one of the following holds:
    \begin{enumerate}[(a)]
        \item $T_k(\varphi)=0$ identically, and $\cD(\varphi)=\cD_{2k}^{\mathrm{odd}}$.
        \item $T_k(\varphi)=T_{k-1}(\varphi)$ identically, and $\cD(\varphi)=\cD_{2k-1}^{\mathrm{even}}$.
        \item $T_k(\varphi)=-T_{k-1}(\varphi)$ identically, and $\cD(\varphi)=\cD_{2k-1}^{\mathrm{odd}}$.
    \end{enumerate}
    In particular, each $T_m(\varphi)$ is either a constant or $\pm T_j(\varphi)$ for some $1\leq j<k$.
\end{lemma}
\deferproof{lem:1-comp}{section,theorem}

Next, we understand the other components of $H_\varphi$.
We first do this ``qualitatively.'' 

\begin{lemma}\label{lem:large-comp}
    Suppose $(\mathcal A,\cB)$ is a $2$-product manifold datum or a $2$-product odd graph datum.
    Select $\varphi\in\cB\setminus\{\1\}$ and let $k=\abs{\im\varphi}$. Suppose that $k\geq 7$.

    In the manifold case, all components of $H_\varphi$ are infinite.

    In the graph case, all components of $H_\varphi$ have at least $k/3$ vertices. Moreover, each component of $H_\varphi$ has some vertex of image size at least $(k/5)^{1/3}$.
\end{lemma}

We then do this ``quantitatively.''
To do so, it is helpful to make an additional assumption on $\varphi$:

\begin{defn}\label{def:primitive}
    An element $\varphi\in\cB$ is \emph{primitive} if, whenever $\psi\in\cB$ is such that $\varphi\in\R[\psi]$, we have $\psi\in\R[\varphi]$.
\end{defn}

We will need a strong form of the existence of primitive eigenfunctions.

\begin{lemma}\label{lem:primitive}
    Suppose $(\mathcal A,\cB)$ is a $2$-product manifold datum or a $2$-product odd graph datum.
    For every $\varphi\in\cB$, there exists some primitive $\varphi'\in\cB$ for which $\varphi\in\R[\varphi']$.
\end{lemma}

For primitive $\varphi$, we can understand the components of $H_\varphi$ quite precisely.

\begin{lemma}\label{lem:comp-labels}
    Suppose $(\mathcal A,\cB)$ is a $2$-product manifold datum or a $2$-product odd graph datum.
    Select $\varphi\in\cB\setminus\{\1\}$ primitive and let $k=\abs{\im\varphi}$. Suppose that $k\geq 258$.

    In the manifold case, each component besides the $\1$-component has all edge labels equal to $1$ and is a bi-infinite path or uni-infinite path with no self loop. The latter case may occur in at most one component, and in such a case the vertex $\psi$ at the end of the uni-infinite path satisfies $\varphi^2+\psi^2=\1$.

    In the graph case, all components besides the $\1$-component are $2k$-cycles with edge labels alternately $1,-1$. Moreover, $k$ is odd and $\cD(\varphi)=\cD_{2k}^{\mathrm{odd}}$.
\end{lemma}
\deferproof{lem:comp-labels}{section,theorem}

Now that we have understood the structure of $H_\varphi$, our proofs in the two settings diverge.
In the manifold setting, we aim for an analogue of \cref{lem: constant angle inner product}, and so we must understand the inner product $\ip{\grad \varphi,\grad\psi}$ for $\varphi,\psi\in\mathcal B$.
Once we have done this, we use an argument similar to that in \cref{sec: 2PP implies flat} to obtain flatness.

In the graph setting, the flavor of argument is quite different. 
The conclusion $\cD(\varphi)=\cD_{2k}^{\mathrm{odd}}$ is enough to tell us that, for each primitive $\varphi\in\mathcal B$ with large enough image size, there is some $\tilde\varphi\in\mathcal B$ for which $\varphi^2+\tilde\varphi^2=\1$ identically. (In fact, $\tilde\varphi=T_{k-1}\varphi$.)
Similarly to how we used \cref{lem: tilde phi} to construct a map from a cover of $M$ to $\mathbb S^1$, we can use the pair $(\varphi,\tilde\varphi)$ to define a map $V(G)\to\mathbb S^1$. We use these maps to generate an abelian group $\Gamma$, and eventually prove that $G$ is a Cayley or Cayley sum graph on the dual group $\widehat\Gamma$.

It turns out that, in both settings, the following computation is useful.

\begin{lemma}\label{lem:doubling}
    Suppose that $\varphi\neq\psi$ are elements of $\mathcal B$ with $\abs{\im\varphi},\abs{\im\psi}\geq 13$, and suppose that
    \[2\varphi\psi=\psi_1+\psi_{-1}\]
    for some distinct $\psi_1,\psi_{-1}\in\cB\cup-\cB$. Then
    \[2\psi_1\psi_{-1}=T_2(\varphi)+T_2(\psi).\]
\end{lemma}
\begin{proof}
    By \cref{lem:large-comp} and our assumption $\abs{\im\varphi}\geq13$, the connected component $C$ of $H_\varphi$ containing $\psi$ must have size at least $5$.
    This implies that $\psi_1$ and $\psi_{-1}$ are orthogonal, since otherwise $C=\{\psi,\psi_1,\psi_{-1}\}$.
    
    Since $\varphi\neq\psi$, neither $\psi_1$ nor $\psi_{-1}$ are constant. We conclude that
    \begin{align*}
        2\varphi\psi_1&=\psi+a_1\psi_2\\
        2\varphi\psi_{-1}&=\psi+a_{-1}\psi_{-2}
    \end{align*}
    for some basis elements $\psi_2,\psi_{-2}\not\in\{\psi,-\psi\}$ and some $a_1,a_{-1}\in\R$ (not necessarily nonzero). 
    If $a_1=a_{-1}=0$, then $C=\{\psi,\psi_1,\psi_{-1}\}$, contradicting $\abs{C}\geq 5$. 
    If neither $a_1=0$ nor $a_{-1}=0$, then we cannot have $\psi_2=\pm\psi_{-2}$: indeed, $2\varphi\psi_2$ has nonzero inner product with the orthogonal basis elements $\psi_1$ and $\psi_{-1}$, and so $C$ consists only of $\{\psi,\psi_1,\psi_{-1},\psi_2\}$. 
    We conclude that either $a_1a_{-1}=0$ or $\int \psi_2\psi_{-2}=0$, and so
    \[\int (2\psi_1\psi_{-1})T_2(\varphi)=\int(2\psi_1\psi_{-1})(2\varphi^2)=\int(\psi+a_1\psi_2)(\psi+a_{-1}\psi_{-2})=\int\psi^2=\frac12.\]
    In particular, $T_2(\varphi)$ is one of the terms in the basis decomposition of $2\psi_1\psi_{-1}$. The symmetry of $\varphi$ and $\psi$ in the lemma statement allows us to conclude also that $T_2(\psi)$ is also a term in the basis decomposition of $2\psi_1\psi_{-1}$. Finally, we have
    \[\langle T_2(\varphi),T_2(\psi)\rangle=\int (2\varphi^2-1)(2\psi^2-1)=\int(\psi_1+\psi_{-1})^2-2\varphi^2-2\psi^2+1=0,\]
    so $T_2(\varphi)$ and $T_2(\psi)$ are orthogonal. The result follows.
\end{proof}

\subsection{Obtaining flatness in the manifold setting}\label{sec:comp-manifold}

Given the previous lemmas, a short computation is enough to prove \cref{prop: 2 product eigenbasis implies flatness} in the manifold setting. 
This proof is mostly independent of the proof we gave in \cref{sec: 2PP implies flat}, except that we refer to the computation of \cref{lem: gradient norm}.

\begin{lemma}\label{lem:const-inner}
    For each $\varphi,\psi\in\cB\setminus\{\1\}$ for which $\varphi$ is primitive, the quantity
    \[
        \bigg\langle\frac{\grad \varphi}{\abs{\grad \varphi}},\frac{\grad \psi}{\abs{\grad \psi}}\bigg\rangle
    \]
    is locally constant away from the critical sets of $\varphi$ and $\psi$.
\end{lemma}
\begin{proof}
    If $\varphi=\psi$, the conclusion is immediate.
    Otherwise, let $\lambda$ and $\mu$ be the eigenvalues corresponding to $\varphi$ and $\psi$, respectively.
    
    First, suppose that we can write
    \[2\varphi\psi=\psi_1+\psi_{-1}\]
    for some $\psi_1,\psi_{-1}\in\cB\cup-\cB$.
    Let $\alpha$ and $\beta$ be the eigenvalues corresponding to $\psi_1$ and $\psi_{-1}$, respectively. 
    First, we compute
    \[\int(2\varphi\psi)^2=\int(\psi_1+\psi_{-1})^2=1.\]
    In particular, $T_2(\varphi)$ and $T_2(\psi)$ are orthogonal. Also, \cref{lem: gradient norm} implies that $\Delta(2\psi^2)=4\mu(2\psi^2-1)$. This gives us
    \begin{align*}
    \int\abs{\nabla(2\varphi\psi)}^2&=\int(2\varphi\psi)(\Delta(2\varphi\psi))\\
    &=\int 2\varphi\psi(2\varphi\psi(\lambda+\mu)-4\ip{\nabla\varphi,\nabla\psi})\\
    &=(\lambda+\mu)\int(2\varphi\psi)^2-2\int\langle\nabla\varphi^2,\nabla\psi^2\rangle\\
    &=(\lambda+\mu)-\int\varphi^2(\Delta\psi^2)=(\lambda+\mu)-4\mu\int\varphi^2(2\psi^2-1)=\lambda+\mu
    \end{align*}
    and
    \[\int\abs{\nabla(2\varphi\psi)}^2=\int\abs{\nabla\psi_1+\nabla\psi_{-1}}^2=\int(\psi_1+\psi_{-1})(\Delta(\psi_1+\psi_{-1}))=\frac{\alpha+\beta}2.\]
    We conclude $\alpha+\beta=2(\lambda+\mu)$. Now, we compute
    \[8\langle\nabla \varphi,\nabla \psi\rangle=4(\lambda+\mu)\varphi\psi-2(\alpha\psi_1+\beta\psi_{-1})=(\alpha+\beta)(\psi_1+\psi_{-1})-2(\alpha\psi_1+\beta \psi_{-1})=-(\alpha-\beta)(\psi_1-\psi_{-1}).\]
    Now, we have by \cref{lem:doubling} that
    \[(\psi_1-\psi_{-1})^2=(\psi_1+\psi_{-1})^2-4\psi_1\psi_{-1}=(2\varphi\psi)^2-2(2\varphi^2+2\psi^2-2)=4(1-\varphi^2)(1-\psi^2).\]
    Therefore, we compute using \cref{lem: gradient norm} that
    \[\left\langle\frac{\grad\varphi}{\abs{\grad\varphi}},\frac{\grad\psi}{\abs{\grad\psi}}\right\rangle^2=\frac{64\langle\nabla\varphi,\nabla\psi\rangle^2}{64\abs{\nabla\varphi}^2\abs{\nabla\psi}^2}=\frac{(\alpha-\beta)^2(\psi_1-\psi_{-1})^2}{64\lambda\mu(1-\varphi^2)(1-\psi^2)}=\frac{(\alpha-\beta)^2(\psi_1-\psi_{-1})^2}{16\lambda\mu(\psi_1-\psi_{-1})^2}.\]
    The conclusion follows.

    Now, by \cref{lem:comp-labels}, the only remaining case is when $\varphi^2+\psi^2=\1$. In this case, $0=\nabla\varphi^2+\nabla\psi^2=2\varphi\nabla\varphi+2\psi\nabla\psi$.
    Thus, at each point of $M$, $\nabla\varphi$ and $\nabla\psi$ are parallel. The conclusion thus follows here as well.
\end{proof}

With \cref{lem:const-inner} in hand, we conclude \cref{prop: 2 product eigenbasis implies flatness} in a very similar way to the  proof given in \cref{sec: 2PP implies flat}.

\begin{proof}[Second proof of \cref{prop: 2 product eigenbasis implies flatness}]
    By the embedding theorem of \cite{BerardBessonGallot1994}, sufficiently many elements of any orthogonal eigenbasis of $L^2(M)$ embed $M$ into Euclidean space.
    Letting $\phi_1,\ldots,\phi_\ell$ be some such elements, we can find $\phi_1',\ldots,\phi_\ell'$ primitive such that $\R[\phi_i]\subset\R[\phi_i']$ for each $i$, so that $\phi_1',\ldots,\phi_\ell'$ embed $M$ into Euclidean space as well.

    Fix $p\in M$ and suppose without loss of generality that, with $n=\dim M$, the vectors $\grad\phi_1',\dots, \grad \phi_n'$ are linearly independent at $p$. They must remain independent on a neighborhood $U$ of $p$. By \cref{lem: gradient norm}, we may define $\alpha_i=\arccos(\phi_i')$ and compute that
    \[
        \grad \alpha_i=-\frac{\grad\phi_i'}{\sqrt{1-(\phi_i')^2}}=-\sqrt{\lambda_i'}\frac{\grad\phi_i'}{\abs{\grad\phi_i'}},
    \]
    where $\Delta\phi_i'=\lambda_i'\phi_i'$. The functions $\alpha_i$ therefore form local coordinates on $U$. By \cref{lem:const-inner}, the coefficients $\ip{\grad\alpha_i,\grad\alpha_j}$ are locally constant, hence the metric is flat on $U$. Since $p$ was arbitrary, the result follows.
\end{proof}

\subsection{Finding eigenfunctions with large image size}
\label{sec:find-varied}

We now turn our attention to the graph setting.
In this subsection, we prove \cref{prop:find-varied}. 

\begin{proofwithclaims}{prop:find-varied}
    Suppose $G$ has $n$ vertices and maximum degree $d$.
    We say a function $\varphi\colon V(G)\to\R$ \emph{respects} a partition $\mathcal P$ of $V(G)$ if $\varphi$ is constant on each part of $\mathcal P$. Any function with image size strictly less than $k$ respects some partition into $k-1$ parts. Given a partition $\mathcal P$ of $V(G)$, let $\Lambda(\mathcal P)$ be the set of $\lambda$ for which $G$ admits an eigenfunction with eigenvalue $\lambda$ respecting $\mathcal P$. For any $\mathcal P$, the set $\Lambda(\mathcal P)$ is finite. 
    We aim to understand the sets $\Lambda(\mathcal P)$; to do so, it is helpful to have an algebraic description of how $\Lambda(\mathcal P)$ can be computed. 

    Suppose $\mathcal P$ has $\ell$ parts for some integer $\ell$, and enumerate the parts $P_1,\ldots,P_\ell$. Let $\chi\colon V(G)\to[\ell]$ be the map such that $v\in P_{\chi(v)}$. Let $\varphi$ be an eigenfunction of some eigenvalue $\lambda$ respecting $\mathcal P$, and for each $1\leq i\leq\ell$ let $x_i$ be the value $\varphi$ takes on $P_i$. The conditions on $x_i$ for $\varphi$ to be an eigenfunction are exactly
    \begin{equation}\label{eq:x-describe}
    \lambda x_{\chi(v)}=\sum_{u\in N(v)}(x_{\chi(v)}-x_{\chi(u)})
    \end{equation}
    for $v\in V(G)$.
    Consider the set  
    \[\mathcal L_\ell:=\left\{\lambda x_i-\sum_{j=1}^\ell a_j(x_i-x_j):\begin{array}{l}
    i\in[\ell]\\
    a_1,\ldots,a_\ell\in\mathbb Z_{\geq 0}\\
    a_1+\cdots+a_\ell\leq d
    \end{array}\right\}\]
    of polynomials in $\lambda,x_1,\ldots,x_\ell$. For each partition $\mathcal P$ and enumeration $P_1,\ldots,P_\ell$ of the parts of $\mathcal P$, there is some subset $\mathcal F\subset\mathcal L_\ell$ such that
    \begin{equation}\label{eq:Lambda-computation}
    \Lambda(\mathcal P)=\big\{\lambda\in\R:\text{there exists }(x_1,\ldots,x_\ell)\in\R^\ell\setminus\{0\}\text{ such that }p(\lambda;x_1,\ldots,x_\ell)=0\text{ for all }p\in\mathcal F\big\}.
    \end{equation}

    We first show that the set $\Lambda(\mathcal P)$ cannot take too many different values.

    \begin{claim}\label{cl:LambdaP-ct}
        As $\mathcal P$ varies among all partitions of $V(G)$ into exactly $\ell$ parts, the set $\Lambda(\mathcal P)$ takes at most $2^{\ell(d+1)^\ell}$ values.
    \end{claim}
    \begin{proofofclaim}
        We have
        \begin{equation}\label{eq:L-size-bound}
        \abs{\mathcal L_\ell}\leq\ell\binom{d+\ell}{d}\leq\ell(d+1)^\ell,
        \end{equation}
        since there are $\ell$ choices of $i$ and $\binom{d+\ell}{d}$ choices of $a_1,\ldots,a_\ell\geq0$ summing to at most $d$.
        Equations \eqref{eq:Lambda-computation} and \eqref{eq:L-size-bound} are enough to prove the claim, as there are $2^{\abs{\mathcal L_\ell}}$ subsets $\mathcal F\subset \mathcal L_\ell$.
    \end{proofofclaim}

    We can also upper-bound the cardinality of each individual $\Lambda(\mathcal P)$.

    \begin{claim}\label{cl:lambdaP-size}
        If $\mathcal P$ is a partition of $V(G)$ into $\ell$ parts, then $\abs{\Lambda(\mathcal P)}\leq\ell$.
    \end{claim}
    \begin{proofofclaim}
       The set $\mathcal L_\ell$ consists of linear polynomials in the $x_i$. 
       We have $\lambda\in\Lambda(\mathcal P)$ if and only if some prescribed collection $\mathcal F$ of these linear polynomials cuts out a nonzero subspace of $\R^\ell$. 
       Form an $\ell\times\abs{\mathcal F}$ matrix whose rows are the coefficients of each polynomial in $\mathcal F$ (some of which are constants and some of which are linear forms in $\lambda$). 
       We have $\lambda\in\Lambda(\mathcal P)$ if and only if this matrix has rank strictly less than $\ell$; that is, if and only if each of the $\ell\times\ell$ submatrices have determinant $0$. 
       The set $\Lambda(\mathcal P)$ is thus cut out by some polynomials of degree at most $\ell$. Since $\Lambda(\mathcal P)$ is finite, these polynomials may not all be zero. 
       Therefore there is at least one such nonzero polynomial, and so $\abs{\Lambda(\mathcal P)}\leq\ell$.
    \end{proofofclaim}

    Claims~\ref{cl:LambdaP-ct}~and~\ref{cl:lambdaP-size} provide an upper bound on the number of eigenvalues which admit eigenfunctions with image size strictly less than $k$. 

    \begin{claim}\label{cl:num-eig}
        A connected graph $G$ on $n$ vertices of maximum degree at most $d$ admits at least $\log_d(n-1)$ distinct eigenvalues.
    \end{claim}
    \begin{proofofclaim}
        Suppose such a graph admits only $M$ distinct eigenvalues. Then its Laplacian matrix $\Delta$ (as it is Hermitian and thus diagonalizable) satisfies some polynomial equation
        \[\Delta^m+b_{m-1}\Delta^{m-1}+\cdots+b_1\Delta+b_0=0.\]
        of some degree $m\leq M$. This implies that the diameter of $G$ is at most $M$: if $u$ and $v$ are vertices connected by a path of length $m$ but no shorter path, then $(\Delta^j)_{uv}=0$ for all $0\leq j<m$ but $(-1)^m(\Delta^m)_{uv}>0$. 
        The claim now follows from the fact that the ball of radius $M$ around a particular vertex $v$ can contain at most $d^M+1$ vertices. 
        So $d^M+1\geq n$, as desired.
    \end{proofofclaim}

    We may now prove \cref{prop:find-varied}. 
    Since $G$ is connected, we may assume $d\geq 2$.
    By Claims~\ref{cl:LambdaP-ct}~and~\ref{cl:lambdaP-size}, the number of eigenvalues of $G$ which admit functions with image size less than $k$ is at most
    \begin{align*}
    \sum_{\ell=1}^{k-1}\ell 2^{\ell(d+1)^\ell}
    &\leq
    2(k-1)2^{(k-1)(d+1)^{k-1}}\\
    &<\frac1{\log d}d^{k-1}\exp\left(d^{2(k-1)}\right)\\
    &=\frac1{\log d}\exp\left(d^{2k-2}+(k-1)\log d\right)<\frac1{\log d}\exp\left(d^{2k-1}\right).
    \end{align*}
    Since we are given that $d<(\log\log n)^{1/2k}$, this quantity is less than $\log(n-1)/\log d$, and thus by \cref{cl:num-eig} is less than the number of distinct eigenvalues of $\Delta$.
    We conclude that, for some eigenvalue $\lambda$ of $\Delta$, every eigenfunction of $\Delta$ with eigenvalue $\lambda$ takes at least $k$ distinct values.
\end{proofwithclaims}

\subsection{Constructing a group from the graph setting}\label{sec:graph-gp}

Suppose that $G$ is a graph with a $2$-product eigenbasis $\mathcal B$ and that $\abs{V(G)}$ is odd. 
Let $\cB_0$ be the set of primitive elements of $\cB$ with image size at least $2^{14}$.
\cref{lem:comp-labels} implies that, if $\varphi\in\cB_0$, then $k$ is odd and $\cD(\varphi)=\cD_{2k}^{\mathrm{odd}}$.
The support of $\cD_{2k}^{\mathrm{odd}}$ is exactly the set $\{\sin(2\pi \ell/k):\ell\in\Z/k\Z\}$.
In particular, there exists some $\alpha_\varphi\colon V(G)\to\frac1k(\Z/k\Z)\subset\R/\Z$ for which
\[\varphi(v)=\sin(2\pi\alpha_\varphi(v)).\]
These $\alpha_\varphi$ (as $\varphi$ varies over all such eigenfunctions) generate (by pointwise addition) an abelian group $\Gamma$ of functions $V(G)\to\R/\Z$, each of which takes rationals with odd denominator as values. 
The group $\Gamma$, as it is generated by finitely many torsion elements of $(\R/\Z)^{V(G)}$, is finite.

Let $\widehat\Gamma$ be the Pontryagin dual of $\Gamma$, viewed as the set of group homomorphisms $\Gamma\to\R/\Z$. The evaluation pairing $\Gamma\times V(G)\to\R/\Z$ gives us a mapping $\xi\colon V(G)\to\widehat\Gamma$ defined by $(\xi(v))(\gamma)=\gamma(v)$.
Consider the functions
\[\operatorname{trig}(\Gamma):=\big\{v\mapsto f(2\pi\gamma(v)):\gamma\in\Gamma,f\in\{\cos,\sin,-\cos,-\sin\}\big\}.\]

The main additional property we need to show is the following.
\begin{lemma}\label{lem:ab-gp}
    Suppose that $\cB$ contains an element with image size at least $2^{45}$.
    Then we have
    \[\{0\}\cup\cB\cup(-\cB)=\operatorname{trig}(\Gamma).\]
    Moreover, $\xi$ is a bijection $V(G)\to\widehat\Gamma$.
\end{lemma}

Our proof of \cref{lem:ab-gp} proceeds in a few steps.
We first need the following simple lemma about the set of rationals with odd denominator.

\begin{lemma}\label{lem:trig-on-odd}
    Let $Q=\{x\in\Q/\Z:x\text{ has odd denominator}\}$. Then we have
    \begin{enumerate}
        \item The function $x\mapsto\sin(2\pi x)$ is injective on $Q$.

        \item The images $\sin(2\pi Q)$ and $\cos(2\pi Q)$ are disjoint.

        \item The images $\cos(2\pi Q)$ and $-\cos(2\pi Q)$ are disjoint.
    \end{enumerate}
\end{lemma}
\begin{proof}
    If $x,y\in\R/\Z$ satisfy $\sin(2\pi x)=\sin(2\pi y)$ while $x\neq y$, then $\cos(2\pi x)=-\cos(2\pi y)$, and so $\cos(2\pi(x+y))=-1$. This implies that (in $\R/\Z$) $x+y=1/2\not\in Q$, which proves (1).

    If $x,y\in\R/\Z$ satisfy $\sin(2\pi x)=\cos(2\pi y)$, then we have $y=\pm(x-1/4)$. We conclude that $x-y=1/4\not\in Q$ or $x+y=1/4\not\in Q$, which proves (2). 

    If $x,y\in Q$ satisfy $\cos(2\pi x)=-\cos(2\pi y)$, then $x\neq y$. By (1), this implies $\sin(2\pi x)\neq\sin(2\pi y)$, so $\sin(2\pi x)=-\sin(2\pi y)$. We conclude that $x-y=1/2\not\in Q$, which proves (3).    
\end{proof}

Next, we need two statements about squares of elements of $\cB$.
One corollary of this lemma is that (with the exception of some eigenfunctions with small image size) the set $\cB$ can be partitioned into pairs $\psi,\psi'$ of elements satisfying $\psi^2+(\psi')^2=\1$.

\begin{lemma}\label{lem:eigsq}
Let $\psi\in\cB$.
\begin{enumerate}
    \item Suppose that $\psi_0\in\cB$ is primitive such that $k:=\abs{\im\psi_0}\geq 7$ and $\psi=\pm T_m(\psi_0)$. Then
    \[\psi^2+(T_{k-m}\psi_0)^2=\1.\]
    In particular, if $\abs{\im\psi}\geq 258$, there is some element $\tilde\psi$ of $\cB$ satisfying $\psi^2+\tilde\psi^2=\1$.
    \item If $\abs{\im\psi}\geq 7$, and additionally some $\psi'\in\cB$ satisfies $\abs{\im\psi'}\geq 7$ and $\psi^2=(\psi')^2$, then $\psi=\psi'$ identically.
\end{enumerate}
\end{lemma}
\begin{proof}
    We first prove (1). \cref{lem:comp-labels} implies that $\cD(\psi_0)=\cD_{2k}^{\mathrm{odd}}$, so $T_k(\psi_0)=0$. Since
    \[T_m^2+T_{k-m}^2-1=T_k\cdot T_{\abs{2m-k}},\]
    we have $T_m(\psi_0)^2+T_{k-m}(\psi_0)^2=\1$, and so the result follows.

    The less specific conclusion in (1) follows by choosing $\psi_0$. 
    For any $\psi\in\cB$, there is some primitive $\psi_0\in\cB$ such that $\psi\in\R[\psi_0]$. Since $\abs{\im\psi_0}\geq\abs{\im\psi}\geq7$, \cref{lem:1-comp} implies that $\psi=\pm T_m(\psi_0)$ for some $m$, and the result follows.

    Now, we prove (2).
    By \cref{lem:1-comp}, every value attained by $\psi$ or $\psi'$ lies in $\cos(2\pi\mathbb Q)$. In particular, every value attained by $2\psi$ or $2\psi'$ is an algebraic integer. If $\psi\not\in\{\psi',-\psi'\}$, then $\psi\perp\psi'$, and so $\sum_{v\in V}\psi(v)\psi'(v)=0$.
    Let $S\subset V(G)$ be the set of vertices for which $\psi(v)=\psi'(v)$, so that
    \[0=\sum_{v\in V(G)}(\psi(v)+\psi'(v))=2\sum_{v\in S}\psi(v)\]
    and
    \begin{align*}
    0=\sum_{v\in V(G)}\psi(v)\psi'(v)
    &=\sum_{v\in S}\psi(v)^2-\sum_{v\in V(G)\setminus S}\psi(v)^2\\
    &=2\sum_{v\in S}\psi(v)^2-\sum_{v\in V(G)}\psi(v)^2=2\sum_{v\in S}\psi(v)^2-\frac{\abs{V(G)}}2.
    \end{align*}
    We conclude that $\sum_{v\in S}2\psi(v)^2=\abs{V(G)}/2$. So
    \[\sum_{\substack{u,v\in S\\u\neq v}}2\psi(u)\psi(v)=2\left(\sum_{v\in S}\psi(v)\right)^2-2\sum_{v\in S}\psi(v)^2=-\frac12\abs{V(G)}.\]
    But the left side is the sum over unordered pairs $\{u,v\}$ of distinct elements of $S$ of $(2\psi(u))(2\psi(v))$, and is thus an algebraic integer. This contradicts the fact that $\abs{V(G)}$ is odd.
\end{proof}

We now take our first major step towards \cref{lem:ab-gp}, which is to give a precise description of the eigenfunctions around each cycle in $H_\varphi$ for primitive $\varphi$.

\begin{lemma}\label{lem:cyc-components-strong}
    Let $\varphi\in\cB$ be primitive and satisfy $k:=\abs{\im\varphi}\geq2^{14}$.
    Let $\psi_0,\psi_1,\ldots,\psi_{2k-1}$ be a $2k$-cycle component of $H_\varphi$.
    Then there exists some $\alpha\colon V(G)\to\R/\Z$ for which, up to choices of sign,
    \[\psi_j=\begin{cases}\cos(2\pi(\alpha+j\alpha_\varphi))&\text{if $j$ is even}\\\sin(2\pi(\alpha+j\alpha_\varphi))&\text{if $j$ is odd}.\end{cases}\]
\end{lemma}

\noindent In particular, \cref{lem:cyc-components-strong} implies that $\psi_j^2+\psi_{j+k}^2=\1$ for all $j$.

\begin{proof}
    By \cref{lem:large-comp}, some $\psi_j$ has image size at least $(k/5)^{1/3}\geq 13$. 
    In fact, we may assume without loss of generality that $\abs{\im\psi_0}\geq13$.
    Then, we may cyclically shift our indexing until $\psi_0$ has image size at least $13$, apply the result in this case, and then shift back, which necessitates replacing $\alpha$ by $\alpha+j\alpha_\varphi$ or $\alpha+j\alpha_\varphi+1/4$ for some $j$.

    By \cref{lem:comp-labels}, we find that we can write
    \begin{align*}
        2\varphi\psi_j&=\psi_{j-1}-\psi_{j+1}&\text{if $j$ is odd;}\\
        2\varphi\psi_j&=-\psi_{j-1}+\psi_{j+1}&\text{if $j$ is even.}
    \end{align*}
    Writing $\beta_j=\psi_j$ if $j$ is even and $\beta_j=i\psi_j$ if $j$ is odd, we find that
    \begin{equation}\label{eq:beta-rec}
        2i\varphi\beta_j=\beta_{j+1}-\beta_{j-1}
    \end{equation}
    for every $j$. Write $\tilde\varphi=\cos(2\pi\alpha_\varphi)$. The characteristic polynomial $x^2-2i\varphi x-1$ factors as
    \[(x-(\tilde\varphi+i\varphi))(x-(-\tilde\varphi+i\varphi))=\left(x-e^{2\pi i\alpha_\varphi}\right)\left(x-e^{-2\pi i(\alpha_\varphi+1/2)}\right).\]
    Write $e(x)=e^{2\pi ix}$ for notational simplicity.
    Solving the recurrence \eqref{eq:beta-rec},\footnote{We are solving this recurrence, as a recurrence relation in $\C$, as it is found upon evaluation at each $v\in V(G)$, and then stitching the resulting data back together in the form of functions from $V(G)$.} we find that, for some $z_1,z_2\colon V(G)\to\C$, we have
    \begin{equation}\label{eq:beta-solve}
    \beta_j=z_1e(j\alpha_\varphi)+z_2e(-j(\alpha_\varphi+1/2)).
    \end{equation}
    Since $\beta_0=z_1+z_2$ is real and $\beta_k=z_1-z_2$ is imaginary, we have $z_1=\overline{z_2}$.
    We now evaluate \eqref{eq:beta-solve} at $j=-1,0,1$:
    \begin{align*}
        \psi_{-1}&=-i\beta_{-1}=iz_2e(\alpha_\varphi)-iz_1e(-\alpha_\varphi)\\
        \psi_0&=\beta_0=z_1+z_2\\
        \psi_1&=-i\beta_1=iz_2e(-\alpha_\varphi)-iz_1e(\alpha_\varphi).
    \end{align*}
    Applying \cref{lem:doubling} to the identity $2\varphi\psi_0=\psi_1-\psi_{-1}$, we have
    \begin{align*}
    \1-\varphi^2
    &=\psi_1\psi_{-1}+\psi_0^2\\
    &=-(z_1^2+z_2^2)+z_1z_2(e(-2\alpha_\varphi)+e(2\alpha_\varphi))+(z_1+z_2)^2\\
    &=z_1z_2(2+e(-2\alpha_\varphi)+e(2\alpha_\varphi))\\
    &=z_1z_2(2\cos(2\pi\alpha_\varphi))^2=4z_1z_2(\1-\varphi^2).
    \end{align*}
    By \cref{lem:trig-on-odd}(2) and the fact that $\alpha_\varphi$ takes only values with odd denominator, we cannot have $\varphi(v)^2=1$ for any $v\in V(G)$. We conclude that $4z_1z_2=\1$ identically. Combined with $z_1=\overline{z_2}$, we have that there is some $\alpha\colon V(G)\to\R/\Z$ for which $z_1=\frac12e(\alpha)$ and $z_2=\frac12e(-\alpha)$. The result now follows from \eqref{eq:beta-solve}.
\end{proof}

We first show that $\cB_0$ is enough to generate all of the eigenfunctions.

\begin{lemma}\label{lem:basis-to-trig}
    Suppose that $\cB$ contains an element of image size at least $2^{45}$.
    Then $\cB\subset\operatorname{trig}(\Gamma)$.
\end{lemma}
\begin{proof}
    Let $\varphi$ be a primitive basis element with $k:=\abs{\im\varphi}\geq 2^{45}$.
    (Such an element of $\cB$ can be found by applying \cref{lem:primitive} to any basis element with image size at least $2^{45}$.)
    Let $\psi$ be any basis element. 
    If $\psi$ lies in the $\1$-component of $H_\varphi$, then $\psi=T_m(\varphi)$ for some $m$, and so
    \[\psi\in\{\pm\cos(2\pi m\alpha_\varphi),\pm\sin(2\pi m\alpha_\varphi)\}\subset\operatorname{trig}(\Gamma).\]

    Otherwise, \cref{lem:comp-labels} implies that $\psi$ lies in a $2k$-cycle of $H_\varphi$. 
    Let $\psi_0$ be an element of this cycle with $\abs{\im\psi_0}\geq 2^{14}$, guaranteed to exist by \cref{lem:large-comp} since $k\geq 2^{45}$. 
    By \cref{lem:cyc-components-strong}, there is some $\alpha\colon V(G)\to\R/\Z$ and some $0\leq j<2k$ for which
    \[\psi_0=\cos(2\pi\alpha);\ \psi_k=\sin(2\pi\alpha);\ \psi\in\{\cos(2\pi(\alpha+j\alpha_\varphi)),\sin(2\pi(\alpha+j\alpha_\varphi)):0\leq j<k\}.\]
    By \cref{lem:primitive}, there exists some primitive $\psi'\in\cB$ for which $\psi_0\in\R[\psi']$; we have $\psi_0=\pm T_m(\psi')$ for some $m$, and $\abs{\im\psi'}\geq\abs{\im\psi_0}\geq258$.
    This gives, writing $\alpha':=\alpha_{\psi'}$,
    \[\psi_0\in\{\pm\cos(2\pi m\alpha'),\pm\sin(2\pi m\alpha')\}.\]
    Let $k_0=\abs{\im\psi'}$, so that $T_{k_0}(\psi')=0$ by \cref{lem:comp-labels}, and write $\ell:=k_0-m$.
    By \cref{lem:eigsq}(1) and the last conclusion in \cref{lem:1-comp}, we have $T_{k_0-m}\psi'\in\cB\cup(-\cB)$ and $\psi_0^2+(T_\ell\psi')^2=\1$. We also have
    \[T_\ell\psi'\in\{\pm\cos(2\pi (k_0-m)\alpha'),\pm\sin(2\pi(k_0-m)\alpha')\}=\{\pm\cos(2\pi m\alpha'),\pm\sin(2\pi m\alpha')\},\]
    and, by \cref{lem:eigsq}(2), we have $T_\ell\psi'=\pm\psi_k$. We conclude 
    \[\psi_0,\psi_k\in\{\pm\cos(2\pi m\alpha'),\pm\sin(2\pi m\alpha')\}.\]
    This implies that $\alpha=\pm m\alpha'+t$ for some sign $\pm$ and some $t\in\{0,1/4,1/2,3/4\}$. Thus
    \[\psi=\cos(2\pi(\pm m\alpha'+j\alpha_\varphi+t'))\]
    for some $t'\in\{0,1/4,1/2,3/4\}$ and some $0\leq j<k$, which is enough to imply the result.
\end{proof}

Next we show that the elements of $\operatorname{trig}(\Gamma)$ are basis elements.

\begin{lemma}\label{lem:trig-to-basis}
    For every $\varphi_1,\ldots,\varphi_m\in\cB_0$ and every $c_1,\ldots,c_m\in\Z$, the functions
    \[\sin\left(2\pi\sum_{i=1}^mc_i\alpha_{\varphi_i}\right)\quad\text{and}\quad\cos\left(2\pi\sum_{i=1}^mc_i\alpha_{\varphi_i}\right)\]
    are, up to sign, basis elements or the constant-$0$ function.
\end{lemma}
\begin{proof}
    Since each $\alpha_\varphi$ has finite order as an element of $(\R/\Z)^{V(G)}$, it suffices to prove the result for $c_1,\ldots,c_m\geq 0$.
    We induct on $s:=c_1+\cdots+c_m$, with the base case of $s=0$ is trivial. To show the inductive step, we have some $\varphi\in\cB_0$ and some $f\colon V(G)\to Q$ for which $\sin(2\pi f)$ and $\cos(2\pi f)$ are both elements of $\{0\}\cup\cB\cup-\cB$, and we must show that the same holds of $\sin(2\pi(\alpha_\varphi+f))$ and $\cos(2\pi(\alpha_\varphi+f))$.

    Consider the component $C$ of $H_\varphi$ containing $\pm\sin(2\pi f)$. 
    (If $\sin(2\pi f)=0$, then take $C$ to be the $\1$-component.)
    If $C$ is the $\1$-component, then $\sin(2\pi f)=\pm T_m(\varphi)$ for some $m$. 
    By \cref{lem:eigsq}(1), with $\psi_0=\varphi$, we have $\cos(2\pi f)=\pm T_{k-m}(\varphi)$. 
    These combine to imply that $f=\pm m\alpha_\varphi+t$ identically for some sign $\pm$ and some $t\in\{0,1/2\}$. So
    \[\sin(2\pi(f+\alpha_\varphi)),\cos(2\pi(f+\alpha_\varphi))\in\{\pm\cos(2\pi(m\pm 1)\alpha_\varphi),\pm\sin(2\pi(m\pm 1)\alpha_\varphi)\}.\]
    The result follows from the identities
    \[\sin(2\pi n\alpha_\varphi)=
    \begin{cases}
        (-1)^{(n-1)/2}T_n(\varphi)&\text{$n$ odd}\\
        (-1)^{(n+k-1)/2}T_{n+k}(\varphi)&\text{$n$ even}\\
    \end{cases}
    \quad\text{and}\quad
    \cos(2\pi n\alpha_\varphi)=
    \begin{cases}
        (-1)^{(n+k)/2}T_{n+k}(\varphi)&\text{$n$ odd}\\
        (-1)^{n/2}T_{n}(\varphi)&\text{$n$ even}\\
    \end{cases}\]
    and the last part of \cref{lem:1-comp}.

    If $C$ is not the $\1$-component, then by \cref{lem:comp-labels} it is a $2k$-cycle. Let $\psi_0,\ldots,\psi_{2k-1}$ be the vertices around the $2k$-cycle in that order, in such a way that $\sin(2\pi f)=\psi_k$. By \cref{lem:cyc-components-strong}, we can find some $\alpha\colon V(G)\to\R/\Z$ for which, up to sign,
    \[\psi_j=\begin{cases}\cos(2\pi(\alpha+j\alpha_\varphi))&\text{if $j$ is even}\\\sin(2\pi(\alpha+j\alpha_\varphi))&\text{if $j$ is odd}.\end{cases}\]
    In particular, $\sin(2\pi f)=\pm\sin(2\pi\alpha)$ identically. Since $\cos(2\pi f)$ is also an element of $\{0\}\cup\cB\cup-\cB$ and both $\psi=\cos(2\pi f)$ and $\psi=\psi_0$ satisfy $\psi^2=\1-\sin(2\pi f)^2$, \cref{lem:eigsq}(2) implies that $\cos(2\pi f)=\pm\cos(2\pi\alpha)$ identically.
    So, we conclude $f=\pm\alpha+t$ identically for some sign $\pm$ and some $t\in\{0,1/2\}$. 
    This implies
    \[\sin(2\pi(f+\alpha_\varphi)),\cos(2\pi(f+\alpha_\varphi))\in\{\pm\psi_{\pm 1},\pm\psi_{k\pm1}\}.\qedhere\]
\end{proof}

We have now shown the first part of \cref{lem:ab-gp}; it remains only to show the second.

\begin{lemma}\label{lem:gamma-leq-N}
    Suppose that $\cB$ contains an element with image size at least $2^{45}$. 
    Then the map $\xi\colon V(G)\to\widehat\Gamma$ is injective.  
\end{lemma}
\begin{proof}
    Take two distinct vertices $u,v\in V(G)$ and consider some basis element $\varphi$ for which $\varphi(u)\neq\varphi(v)$.
    By \cref{lem:basis-to-trig}, there exists some $\gamma\in\Gamma$ for which $\varphi=\pm\sin(2\pi\gamma)$ or $\varphi=\pm\cos(2\pi\gamma)$. In any case, we have $\gamma(u)\neq\gamma(v)$. This implies that $\xi(u)\neq\xi(v)$ (as these two elements of $\widehat\Gamma$ do not agree on $\gamma$), as desired.
\end{proof}

\begin{lemma}\label{lem:gamma-geq-N}
    We have $\abs{\Gamma}\leq\abs{V(G)}$.
\end{lemma}
\begin{proof}
    We will prove this inequality by showing that the functions $\cos(2\pi\gamma)$ and $\sin(2\pi\gamma)$ cannot collide in ``unexpected'' ways as $\gamma$ ranges through $\Gamma$.
    
    \begin{claim}\label{cl:non-const}
        For nontrivial $\gamma\in\Gamma$, the functions $\sin(2\pi\gamma)$ and $\cos(2\pi\gamma)$ are nonconstant.
    \end{claim}
    \begin{proofofclaim}
        \cref{lem:trig-to-basis} implies that any such constant function must be identically $0$ or $\pm1$. Both are ruled out by the assumption that $\gamma\colon V(G)\to\R/\Z$ is not identically zero and the fact that $\gamma$ cannot take rationals with even denominators as values.
    \end{proofofclaim}

    Now, recall that $\Gamma$ has odd order. 
    The set $\Gamma$ can be partitioned into the identity element $\{0\}$ and $(\abs{\Gamma}-1)/2$ pairs $\{\gamma,-\gamma\}$. 
    There is a map $\rho$ from pairs $\{\gamma,-\gamma\}$ to quadruplets
    \[\{\cos(2\pi\gamma),\sin(2\pi\gamma),-\cos(2\pi\gamma),-\sin(2\pi\gamma)\}\]
    of non-constant elements of $\cB$.
    \begin{claim}\label{cl:rho-inj}
        The quadruplet $\rho(\{\gamma,-\gamma\})$ consists of four distinct elements of $\cB\cup-\cB$. 
    \end{claim}
    \begin{proofofclaim}
        We cannot have $\cos(2\pi\gamma)=-\cos(2\pi\gamma)$ or $\sin(2\pi\gamma)=-\sin(2\pi\gamma)$ identically, as such an identity would contradict \cref{cl:non-const}. 
        If $\pm\cos(2\pi\gamma)=\sin(2\pi\gamma)$, then $\cos(2\pi\gamma)^2=1/2$ identically. This implies that the image of $\gamma$ is contained in $\{1/8,3/8,5/8,7/8\}$, which contradicts the fact that no element of $\Gamma$ takes as a value a rational with even denominator.
    \end{proofofclaim}

    \begin{claim}
        For distinct pairs $\{\gamma,-\gamma\}$ and $\{\gamma',-\gamma'\}$, the quadruplets $\rho(\{\gamma,-\gamma\})$ and $\rho(\{\gamma',-\gamma'\})$ are disjoint.
    \end{claim}
    \begin{proofofclaim}
        Up to the signs and swapping of $\gamma$ and $\gamma'$, there are four cases:
        \begin{align}
            \label{eq:ss}\sin(2\pi\gamma)&=\sin(2\pi\gamma')\\
            \label{eq:sc}\sin(2\pi\gamma)&=\cos(2\pi\gamma')\\   
            \label{eq:cmc}\cos(2\pi\gamma)&=-\cos(2\pi\gamma')\\
            \label{eq:cpc}\cos(2\pi\gamma)&=\cos(2\pi\gamma').
        \end{align}
        By \cref{lem:trig-on-odd}(1), \eqref{eq:ss} implies that $\gamma(v)=\gamma'(v)$ for every $v\in V(G)$. This contradicts the assumption $\gamma\neq\gamma'$. Cases \eqref{eq:sc} and \eqref{eq:cmc} cannot occur by \cref{lem:trig-on-odd}(2 and 3).
        
        Finally, if \eqref{eq:cpc} holds, then we must have $\gamma'(v)\in\{\gamma(v),-\gamma(v)\}$ for every $v\in V(G)$. Consider the identity
        \begin{equation}\label{eq:trig-two-decomp}
            1+\cos(4\pi\gamma)=2\cos(2\pi\gamma)^2=2\cos(2\pi\gamma)\cos(2\pi\gamma')=\cos(2\pi(\gamma+\gamma'))+\cos(2\pi(\gamma-\gamma')).
        \end{equation}
        By \cref{lem:trig-to-basis}, the identity \eqref{eq:trig-two-decomp} witnesses two decompositions of the same function into linear combinations of elements of $\cB$. 
        Therefore, at least one of $\{\cos(4\pi\gamma),\cos(2\pi(\gamma+\gamma')),\cos(2\pi(\gamma-\gamma'))\}$ is constant.
        By \cref{cl:non-const}, we either have $2\gamma=0$ or $\gamma=\pm\gamma'$ identically, neither of which can hold by the fact that $\abs{\Gamma}$ is odd and our assumption that $\{\gamma,-\gamma\}\neq\{\gamma',-\gamma'\}$.
    \end{proofofclaim}

    We conclude that $\cB\cup-\cB$ has at least $2\abs{\Gamma}-2$ non-constant elements. Since the number of such elements is exactly $2\abs{V(G)}-2$, we obtain the desired conclusion.
\end{proof}

\cref{lem:ab-gp} follows from \cref{lem:trig-to-basis,lem:basis-to-trig,lem:gamma-geq-N,lem:gamma-leq-N}. 

\subsection{Finishing the proof in the graph setting}\label{sec:graph-finish}

We are now ready to prove \cref{thm: 2PP iff Cayley} assuming \cref{lem:1-comp,lem:large-comp,lem:primitive,lem:comp-labels}.

\begin{proof}[Proof of \cref{thm: 2PP iff Cayley}]
    Suppose that $G$ admits a $2$-product eigenbasis, and that $\abs{V(G)}$ is odd and sufficiently large in terms of the maximum degree of $G$.
    We must show that $G$ is either a Cayley graph or a Cayley sum graph on an abelian group.

    Let $\cB$ be a $2$-product eigenbasis.
    By \cref{prop:find-varied}, we may assume that $\cB$ contains some element with image size at least $2^{45}$.
    So, we may apply \cref{lem:ab-gp} to find an abelian group $\Gamma$, which we view as a group of functions $V(G)\to\R/\Z$, such that (1) there is a bijection $\xi\colon V(G)\to \widehat\Gamma$ satisfying $(\xi(v))(\gamma)=\gamma(v)$ and (2) the elements of $\cB$ are exactly, up to sign, the trigonometric functions
    \[v\mapsto \sin(2\pi \gamma(v))\quad\text{and}\quad v\mapsto\cos(2\pi\gamma(v))\]
    for $\gamma\in\Gamma$.

    Let $\Gamma_0$ be a subset of $\Gamma$ of size $\frac{\abs{\Gamma}-1}2$ consisting of an arbitrary choice of one element of $\{\gamma,-\gamma\}$ for each $\gamma\in\Gamma\setminus\{0\}$. For each $\gamma\in\Gamma_0$, let $\mu_\gamma^{\cos}$ and $\mu_\gamma^{\sin}$ be the eigenvalues of the Laplacian $\Delta$ corresponding to $\cos(2\pi\gamma)$ and $\sin(2\pi\gamma)$. 
    The spectral theorem tells us that, for $u,v\in V(G)$, we have, writing $N:=\abs{V(G)}$,
    \begin{align*}
        \Delta_{uv}
        &=\frac 1N\sum_{\gamma\in\Gamma_0}\mu_\gamma^{\cos}2\cos(2\pi\gamma(u))\cos(2\pi\gamma(v))+\mu_\gamma^{\sin}2\sin(2\pi\gamma(u))\sin(2\pi\gamma(v))\\
        &=\frac 1N\sum_{\gamma\in\Gamma_0}(\mu_\gamma^{\cos}+\mu_\gamma^{\sin})\cos\big(2\pi(\gamma(u)-\gamma(v))\big)+\frac 1N\sum_{\gamma\in\Gamma_0}(\mu_\gamma^{\cos}-\mu_\gamma^{\sin})\cos\big(2\pi(\gamma(u)+\gamma(v))\big)\\
        &=\frac 1N\sum_{\gamma\in\Gamma_0}(\mu_\gamma^{\cos}+\mu_\gamma^{\sin})\cos\big(2\pi(\xi(u)-\xi(v))(\gamma)\big)+\frac 1N\sum_{\gamma\in\Gamma_0}(\mu_\gamma^{\cos}-\mu_\gamma^{\sin})\cos\big(2\pi(\xi(u)+\xi(v))(\gamma))\big).
    \end{align*}
    In particular, there exist even functions $f,g\colon\widehat\Gamma\to\R$ for which
    \[\Delta_{uv}=f(\xi(u)-\xi(v))+g(\xi(u)+\xi(v))\]
    for every $u,v\in V(G)$.
    
    Since $\Gamma$ (and thus $\widehat\Gamma$) is of odd order and $\xi$ is a bijection, we can find for every $x,y\in\widehat\Gamma$ some $u,v\in V(G)$ for which $x=\xi(u)-\xi(v)$ and $y=\xi(u)+\xi(v)$.
    Since $G$ is an unweighted graph, the entries of $\Delta$ are integers, and the non-diagonal entries of $\Delta$ all lie in the set $\{-1,0\}$.
    Letting $U_f:=\{f(x):x\in\widehat\Gamma\setminus\{0\}\}$ and $U_g:=\{g(y):y\in\widehat\Gamma\}$, we have $U_f+U_g\subset\{-1,0\}$. Therefore, either $U_f$ or $U_g$ (or both) is a singleton. By replacing $(f,g)$ with $(f+a,g-a)$ for some real $a$, we can assume $U_f=\{0\}$ (and hence $U_g\subset\{-1,0\}$) or $U_g=\{0\}$ (and hence $U_f\subset\{-1,0\}$).

    If $U_f=\{0\}$, let $S\subset\widehat\Gamma$ be the support of $g$.
    Since $g$ is even, $S$ is symmetric. 
    Since $\Delta\1=0$, we have $f(0)=\abs{S}$.
    We conclude that $G\cong\Cay^+(\widehat\Gamma,S)$. If, alternatively, $U_g=\{0\}$, let $S=(\supp f)\setminus\{0\}$.
    Since $f$ is even, $S$ is symmetric. 
    We conclude that $G\cong\Cay(\widehat\Gamma,S)$.
\end{proof}

\vspace{3em}

\section{The details of the graph setting}\label{sec:graph-detail}

In this section, we finish the proof of \cref{thm: 2PP iff Cayley} by proving \cref{lem:1-comp,lem:large-comp,lem:primitive,lem:comp-labels}.

\subsection{The moment problem}

To prove some parts of \cref{lem:1-comp}, we will need to be able to compute the distribution of eigenfunctions in both the graph and manifold settings. We will access such data by first computing the moments of such functions and then arguing that these moments uniquely determine the distribution.

\begin{theorem}[{Carleman's criterion; see \cite[Section~4.2]{moments}}]\label{thm:carleman}
    Let $X$ and $Y$ be real-valued random variable with the same finite moments $m_n:=\EE X^n=\EE Y^n$.
    If the series
    \[\sum_{n=1}^\infty m_{2n}^{-1/(2n)}\]
    diverges, then $X$ and $Y$ have the same distribution.
\end{theorem}

\begin{lemma}\label{lem:moment-method}
    Let $X$ be a random variable. Suppose that
    \[\EE T_j(X)=0\]
    for every positive integer $j$. Then $X$ has the same distribution as $\cos(2\pi Y)$ with $Y\sim\operatorname{Unif}(\R/\Z)$.
\end{lemma}
\begin{proof}
    We first note that
    \begin{equation}\label{eq:moment-id}
    \EE T_j(\cos(2\pi Y))=\EE \cos(2\pi jY)=\int_0^1\cos(2\pi jy)dy=0=\EE T_j(X).
    \end{equation}
    for every positive integer $j$.
    Since the polynomials $T_j(x)$ span $\R[x]$, we conclude from taking linear combinations of \eqref{eq:moment-id} that
    \[\EE X^j=\EE \cos(2\pi Y)^j\]
    for every nonnegative integer $j$. (We also use here the trivial fact that $X$ and $\cos(2\pi Y)$ have the same zero-th moment.)
    We may explicitly compute
    \[\EE\cos(2\pi Y)^j=\begin{cases}2^{-j}\binom j{j/2}&\text{if }2\mid j\\0&\text{otherwise.}\end{cases}\]
    \cref{thm:carleman} thus tells us that these moments uniquely determine the distribution $\cos(2\pi Y)$, and the result follows.
\end{proof}

\begin{lemma}\label{lem:discrete-moment-method}
    Let $S$ be a finite set and let $X$ and $Y$ be two probability distributions whose supports are contained in $S$.
    Suppose that
    \[\EE T_j(X)=\EE T_j(Y)\]
    for each $1\leq j<\abs{S}$.
    Then $X$ and $Y$ have the same distribution.
\end{lemma}
\begin{proof}
    Let $n=\abs{S}$ and enumerate $S=\{x_1,\ldots,x_n\}$.
    For $1\leq i\leq n$, let $a_i$ (resp.\ $b_i$) be the probability that $X=x_i$ (resp.\ $Y=x_i$). 
    Since the polynomials $T_0,T_1,\ldots,T_{n-1}$ span the space of polynomials of degree at most $n-1$, we have
    \[\sum_{i=1}^na_ix_i^j=\EE X^j=\EE Y^j=\sum_{i=1}^nb_ix_i^j\]
    for each $0\leq j<n$. Since the Vandermonde matrix $(x_i^j)_{1\leq i\leq n,0\leq j<n}$ is non-singular, we conclude that $a_i=b_i$ for all $i$, as desired.
\end{proof}

We will apply \cref{lem:discrete-moment-method} to $Y\sim\cD_n^{\mathrm p}$ for various $n$ and $\mathrm p$; to this end, we compute its ``Chebyshev moments.'' 

\begin{lemma}\label{lem:nice-moments}
    The distributions $\cD_n^{\mathrm p}$ satisfy
    \[\EE_{X\sim\cD_n^{\mathrm p}}T_j(X)=0\]
    for each $1\leq j<n$.
\end{lemma}
\begin{proof}
    Let $Z$ denote a uniformly random complex solution to $z^n=(-1)^{\mathrm p}$. We have $X\sim\Re Z$. Moreover, since $Z$ is supported on the unit circle, we have $T_j(X)\sim\Re Z^j$. Since the distribution of $Z$ is symmetric under rotation by $e^{2\pi i/n}$, we have $\EE Z^j=0$ for all $1\leq j<n$, from which the result follows.
\end{proof}

\subsection{The \texorpdfstring{$\1$}{1}-component}

Here, we prove \cref{lem:1-comp}.
We begin by providing a simple observation about functions with distribution $\cD_n^{\mathrm p}$ which will enable us to utilize the assumption that $\abs{V(G)}$ is odd.

\begin{obs}\label{obs:div}
    If there is a function $f\colon V(G)\to\R$ for which $\cD(f)=\cD_n^{\mathrm{even}}$, then $n$ divides $\abs{V(G)}$. 

    If, instead, there is a function $f\colon V(G)\to\R$ for which $\cD(f)=\cD_n^{\mathrm{odd}}$, then $n$ odd implies that $n$ divides $\abs{V(G)}$, while $n$ even implies that $n/2$ divides $\abs{V(G)}$.
\end{obs}

Let $\varphi\in\mathcal B\setminus\{\1\}$ be an element of a $2$-product eigenbasis taking $k$ values, where we write $k=\infty$ in the manifold setting. 
(Since our manifolds are assumed to be connected and all eigenfunctions are continuous, every eigenfunction besides the constant function $\1$ in the manifold setting takes infinitely many values.)
Recall that we have assumed $k\geq 7$, and wish to describe the $\1$-component of the auxiliary graph $H_\varphi$.

\begin{proofwithclaims}{lem:1-comp}
    By \cref{obs:max-deg-2}, the $\1$-component of $H_\varphi$ is either a finite path with no self-loops, a finite path with a self-loop at one end, or a uni-infinite path with no self-loop at the finite end.

    One end of this path (without self-loop) is the vertex $\varphi_0:=\1$; let the other vertices be $\varphi=\varphi_1,\ldots,\varphi_\ell$ in order, normalized such that $\norm{\varphi_j}_{L^2}^2=1/2$ for each $1\leq j\leq \ell$. (By abuse of notation, we write $\ell=\infty$ if the $\1$-component is a uni-infinite path.)
    
    For some $a_1,\ldots,a_{\ell-1}\in\R\setminus\{0\}$ and $a_\ell\in\R$, we have
    \begin{align}
        \label{eq:phi-0}
        2\varphi\varphi_0&=2\varphi\\
        \label{eq:phi-1}
        2\varphi\varphi_1&=1+a_1\varphi_2\\
        \label{eq:phi-2}
        2\varphi\varphi_2&=a_1\varphi_1+a_2\varphi_3\\
        \notag&\vdots\\
        \label{eq:phi-ell-1}
        2\varphi\varphi_{\ell-1}&=a_{\ell-2}\varphi_{\ell-2}+a_{\ell-1}\varphi_\ell\\
        \label{eq:phi-ell}2\varphi\varphi_\ell&=a_{\ell-1}\varphi_{\ell-1}+a_\ell\varphi_\ell.
    \end{align}
    By successively swapping the signs of $\varphi_j$ for $j=2,\ldots,\ell$, we can enforce that $a_1,a_2,\ldots,a_{\ell-1}>0$. By induction on $j$, we see that $\varphi_j$ is a polynomial in $\varphi$ of degree exactly $j$.

    We first observe that we can compute $\ell$.
    
    \begin{claim}\label{cl:1-comp-len}
        If $k<\infty$, then $\ell=k-1$.
        If $k=\infty$, then $\ell=\infty$.
    \end{claim}
    \begin{proofofclaim}
    Since $\varphi_0,\ldots,\varphi_\ell$ are linearly independent polynomials in $\varphi$ of degree $0,1,\ldots,\ell$, the function $\varphi$ does not satisfy a polynomial equation of degree strictly less than $\ell+1$. Conversely, \eqref{eq:phi-ell} gives a polynomial equation of degree $\ell+1$ which $\varphi$ satisfies, so $\varphi$ can take at most $\ell+1$ values.
    \end{proofofclaim}

    \begin{claim}\label{cl:1-comp-a}
        Suppose that $\ell\geq 6$.
        Then we have $a_1=\cdots=a_{\ell-2}=1$.
    \end{claim}
    \begin{proofofclaim}
        Let $j$ be a positive integer. Using \eqref{eq:phi-1}, we have $\varphi_2=a_1^{-1}(2\varphi^2-1)$. Let $3\leq j\leq\ell-2$. We may compute
        \begin{align*}
        2a_1\varphi_2\varphi_j
        =2(2\varphi^2-1)\varphi_j
        &=2\varphi(2\varphi\varphi_j)-2\varphi_j\\
        &=2\varphi(a_{j-1}\varphi_{j-1}+a_j\varphi_{j+1})-2\varphi_j\\
        &=a_{j-1}(a_{j-2}\varphi_{j-2}+a_{j-1}\varphi_j)+a_j(a_j\varphi_j+a_{j+1}\varphi_{j+2})-2\varphi_j\\
        &=a_{j-1}a_{j-2}\varphi_{j-2}+(a_{j-1}^2+a_j^2-2)\varphi_j+a_ja_{j+1}\varphi_{j+2}.
        \end{align*}
        We can perform a similar computation for $j=2$; the result reads
        \[2a_1\varphi_2^2=a_1+(a_1^2+a_2^2-2)\varphi_2+a_2a_3\varphi_4.\]
        Applying the $2$-product property of $\cB$ to each of these identities, we must have $a_{j-1}^2+a_j^2=2$ for each $2\leq j\leq\ell-2$. In particular, since $a_j>0$ for each $1\leq j<\ell$, we have $a_1=a_3=a_5=\cdots$ and $a_2=a_4=\cdots$. We may also compute
        \begin{align*}
        2a_1a_2\varphi_3^2
        &=2a_1(2\varphi\varphi_2-a_1\varphi)\varphi_3\\
        &=2\varphi(2a_1\varphi_2\varphi_3)-a_1^2(2\varphi\varphi_3)\\
        &=2\varphi(a_1a_2\varphi+a_3a_4\varphi_5)-a_1^2(a_2\varphi_2+a_3\varphi_4)\\
        &=a_1a_2(1+a_1\varphi_2)+a_3a_4(a_4\varphi_4+a_5\varphi_6)-a_1^2a_2\varphi_2-a_1^2a_3\varphi_4\\
        &=a_1a_2+(a_3a_4^2-a_1^2a_3)\varphi_4+a_3a_4a_5\varphi_6.
        \end{align*}
        Applying the $2$-product property of $\cB$, we obtain $a_3a_4^2=a_1^2a_3$, and thus $a_1=a_4$. We conclude that $a_1=a_2=\cdots=a_{\ell-2}=1$, as desired.
    \end{proofofclaim}

    Using \cref{cl:1-comp-a}, we can show by a simple induction that $\varphi_j=T_j(\varphi)$ for each $1\leq j\leq\ell-1$.
    In the manifold setting, this is enough to conclude the proof of \cref{lem:1-comp}; the random variable $X=\varphi(x)$ for $x$ chosen according the probability measure arising from the volume form satisfies $\EE T_j(X)=0$ for every $j\geq1$, and so the fact that $\cD(\varphi)=\cD^{\mathrm{cont}}$ follows from \cref{lem:moment-method}.
    
    In the graph setting, however, there is more to do. In particular, we must understand $a_{\ell-1}$ and $a_\ell$. We compute, similarly to the proof of \cref{cl:1-comp-a}, that
    \begin{align*}
    2\varphi_2\varphi_{\ell-1}
    &=a_{\ell-3}a_{\ell-2}\varphi_{\ell-3}+(a_{\ell-2}^2+a_{\ell-1}^2-2)\varphi_{\ell-1}+a_{\ell-1}a_\ell\varphi_\ell\\
    &=\varphi_{\ell-3}+(a_{\ell-1}^2-1)\varphi_{\ell-1}+a_{\ell-1}a_\ell\varphi_\ell.
    \end{align*}
    The $2$-product property thus implies that either $a_{\ell-1}=1$ or $a_\ell=0$. We additionally have
    \begin{align*}
        2\varphi_2\varphi_\ell
        &=2\varphi(2\varphi\varphi_\ell)-2\varphi_\ell\\
        &=2\varphi(a_{\ell-1}\varphi_{\ell-1}+a_\ell\varphi_\ell)-2\varphi_\ell\\
        &=a_{\ell-1}(a_{\ell-2}\varphi_{\ell-2}+a_{\ell-1}\varphi_\ell)+a_\ell(a_{\ell-1}\varphi_{\ell-1}+a_\ell\varphi_\ell)-2\varphi_\ell\\
        &=a_{\ell-1}a_{\ell-2}\varphi_{\ell-2}+a_{\ell-1}a_\ell\varphi_{\ell-1}+(a_{\ell-1}^2+a_\ell^2-2)\varphi_\ell.
    \end{align*}
    The $2$-product property here implies that $a_\ell=0$ or $a_{\ell-1}^2+a_\ell^2=2$. 
    Putting these together, we conclude that either $a_\ell=0$ or that $a_{\ell-1}=1$ and $a_\ell=\pm 1$. If $a_\ell=0$, write $b=a_{\ell-1}^2-2$ for notational simplicity. We compute
    \begin{align*}
    2\varphi_3\varphi_{\ell-1}
    &=2\varphi(2\varphi_2\varphi_{\ell-1}-\varphi_{\ell-1})\\
    &=2\varphi(\varphi_{\ell-3}+b\varphi_{\ell-1})\\
    &=a_{\ell-4}\varphi_{\ell-4}+(a_{\ell-3}+a_{\ell-2}b)\varphi_{\ell-2}+a_{\ell-1}b\varphi_\ell\\
    &=\varphi_{\ell-4}+(b+1)\varphi_{\ell-2}+a_{\ell-1}b\varphi_\ell.
    \end{align*}
    The $2$-product property thus gives that $b\in\{-1,0\}$, whence $a_{\ell-1}\in\{1,\sqrt2\}$. We conclude that
    \[(a_{\ell-1},a_\ell)\in\{(1,0),(\sqrt2,0),(1,1),(1,-1)\}.\]
    Three of these four cases will yield the cases in the lemma statement; the fourth will not occur.
    \begin{enumerate}[(a)]
        \item Suppose $a_{\ell-1}=1$ and $a_\ell=0$. Then $\varphi_\ell=T_\ell(\varphi)$, and \eqref{eq:phi-ell} gives that
        \[2\varphi T_\ell(\varphi)=T_{\ell-1}(\varphi)\implies T_{\ell+1}(\varphi)=0.\]
        By \cref{cl:1-comp-len}, this is equivalent to $T_k(\varphi)=0$.

        To show the distributional claim, we note that $\cD(\varphi)$ is supported on the zero set of $T_k$. This set of roots is exactly the support of $\cD_{2k}^{\mathrm{odd}}$. Moreover, for each $1\leq j<k$, we have
        \[\int T_j(\varphi)=\langle T_j(\varphi),\1\rangle=0.\]
        We conclude from \cref{lem:discrete-moment-method,lem:nice-moments} that $\cD(\varphi)=\cD_{2k}^{\mathrm{odd}}$.

        \item Suppose $a_{\ell-1}=a_\ell=1$. Then $\varphi_\ell=T_\ell(\varphi)$, and \eqref{eq:phi-ell} gives that
        \[T_{\ell-1}(\varphi)+T_{\ell+1}(\varphi)=2\varphi T_\ell(\varphi)=T_{\ell-1}(\varphi)+T_\ell(\varphi),\]
        and so $T_{\ell+1}(\varphi)=T_\ell(\varphi)$. 
        By \cref{cl:1-comp-len}, this is equivalent to $T_k(\varphi)=T_{k-1}(\varphi)$.
        
        We thus have that $\cD(\varphi)$ is supported on the zero set of $T_k-T_{k-1}$, which is exactly the set of values $\cos(2\pi t/(2k-1))$ for $0\leq t<2k-1$. This set of roots is exactly the support of $\cD_{2k-1}^{\mathrm{even}}$. Moreover, we have $\int T_j(\varphi)=0$ for all $1\leq j<k$.
        We conclude from \cref{lem:discrete-moment-method,lem:nice-moments} that $\cD(\varphi)=\cD_{2k-1}^{\mathrm{even}}$.

        \item Suppose $a_{\ell-1}=1$ and $a_\ell=-1$. Then $\varphi_\ell=T_\ell(\varphi)$, and \eqref{eq:phi-ell} gives that
        \[T_{\ell-1}(\varphi)+T_{\ell+1}(\varphi)=2\varphi T_\ell(\varphi)=T_{\ell-1}(\varphi)-T_\ell(\varphi),\]
        and so $T_k(\varphi)=-T_{k-1}(\varphi)$.

        We thus have that $\cD(\varphi)$ is supported on the zero set of $T_k+T_{k-1}$. If we had instead run the proof for $-\varphi$ instead of $\varphi$, we would instead have ended up in the previous case. Applying the conclusion in that case and negating $\varphi$ gives the conclusion in this case.

        \item Suppose $a_{\ell-1}=\sqrt2$ and $a_\ell=0$. Then $\varphi_\ell=(1/\sqrt2)T_\ell(\varphi)$, and \eqref{eq:phi-ell} gives that
        \[T_{\ell+1}(\varphi)+T_{\ell-1}(\varphi)=2\varphi T_\ell(\varphi)=2\sqrt2\varphi\varphi_\ell=2T_{\ell-1}(\varphi),\]
        and so $T_k(\varphi)=T_{k-2}(\varphi)$.

        We thus have that $\cD(\varphi)$ is supported on the zero set of $T_k-T_{k-2}$, which is exactly the set of values $\cos(2\pi t/(2k-2))$ for $0\leq t<2k-2$. Moreover, we have $\int T_j(\varphi)=0$ for all $1\leq j<k$. We conclude from \cref{lem:discrete-moment-method,lem:nice-moments} that $\cD(\varphi)\sim\cD_{2k-2}^{\mathrm{even}}$. Since $\abs{V(G)}$ is odd, \cref{obs:div} implies that this cannot occur. \qedhere
    \end{enumerate}
Finally, we show that $T_m(\varphi)$ is either $0$, $\pm\1$, or $\pm T_j(\varphi)$ for some $1\leq j<k$. Call this property $(\star)$.
Let $\mathrm p\in\{\mathrm{odd},\mathrm{even}\}$ and $n\in\{2k-1,2k\}$ be such that $\cD(\varphi)=\cD_n^{\mathrm p}$. We have some function
\[f\colon V(G)\to\{j\in[0,2n):j\text{ of parity }\mathrm p\}\]
such that $\varphi(v)=\cos(\pi f(v)/n)$.

For each positive integer $m$, we have
\begin{equation}\label{eq:iter-cos}
    (T_m\varphi)(v)=\cos(\pi mf(v)/n).
\end{equation}
If $\mathrm p=\mathrm{even}$, the right side of \eqref{eq:iter-cos} is periodic in $m$ with period $n$, while if $\mathrm p=\mathrm{odd}$, the right side of \eqref{eq:iter-cos} is periodic in $m$ with period $2n$ and is negated by adding $n$ to $m$. 
So, in either case, it suffices to prove $(\star)$ for one value $m$ in each residue class modulo $n$.
The result follows from combining the facts that (i) $(\star)$ clearly holds for $0\leq m<k$, (ii) $(\star)$ holds for $m=k$ because $T_k(\varphi)\in\{0,T_{k-1}(\varphi),-T_{k-1}(\varphi)\}$, (iii) $T_m(\varphi)=T_{-m}(\varphi)$ for every $m$, and (iv) $n\leq 2k$.
\end{proofwithclaims}

We now record an additional fact about the distributions $\cD_n^{\mathrm p}$, and its corresponding consequences for eigenfunctions with large image size, which will prove useful later.

\begin{lemma}\label{lem:1-comp-cor}
    Suppose we are in the graph setting, and let $\varphi\in\cB$ have image size $k\geq 7$. 
    Then the maximum point probability of $\cD(\varphi)$ is at most $1/(k-1)$. 
\end{lemma}
\begin{proof}
    This follows from \cref{lem:1-comp}, the explicit descriptions of $\cD_n^{\mathrm p}$ and the fact that $t\mapsto \cos(t\pi/n)$ takes each value at most twice as $t$ ranges through a residue class modulo $2n$.
\end{proof}

\subsection{Degenerate possibilities}

Here, we prove \cref{lem:large-comp}.
This subsection is distinguished from the following subsection in that the arguments are primarily ``soft'' and do not require much of the global structure of the components of $H_\varphi$.
Given $T_\ell(\varphi)\in\cB$ for each $\ell<\abs{\im\varphi}$ the arguments in this section run without issue (albeit with worse quantitative behavior) when generalized to the $N$-product setting.

We begin by dispatching with the manifold case of \cref{lem:large-comp}, which is fairly simple and contains some of the ideas used in the remainder of the section.

\begin{proof}[Proof of \cref{lem:large-comp}, manifold case]
    We must show that all components of $H_\varphi$ are infinite. Suppose for the sake of contradiction that $H_\varphi$ has a finite component with vertices $\psi_1,\ldots,\psi_\ell$.
    Pointwise multiplication by $\varphi$ is a linear map preserving the vector space $\R\psi_1\oplus\cdots\oplus\R\psi_\ell$.
    We conclude that, for some polynomial $Q$ of degree $\ell$ (the characteristic polynomial of this linear map), we have $Q(\varphi)\psi_1=0$ identically, so either $Q(\varphi)$ or $\psi_1$ must vanish on some open set.
    The function $\psi_1$, as it is a nonzero eigenfunction, cannot vanish on an open set, so we have $Q(\varphi)=0$. Since the image of $\varphi$ is a closed interval, we thus have that $\varphi$ is constant. This cannot occur, and so we are done.
\end{proof}

We now restrict attention to the graph case.
We proceed with the help of some auxiliary lemmas.

\begin{lemma}\label{lem:varied-or-large}
    Let $\varphi\in\cB$ have image size $k$.
    Let $m\geq2$ be a positive integer and suppose that $k>4m^3+1$. Then each connected component of $H_\varphi$ is either a singleton or contains some eigenfunction of image size more than $m$.
\end{lemma}
\begin{proof}
    Suppose that we have a non-singleton component $C$ in which every vertex has image size at most $m$.
    Let $\psi_0$ be any such vertex and write $2\varphi\psi_0=a_1\psi_1+a_{-1}\psi_{-1}$. Suppose that $\psi_{-1},\psi_0,\psi_1$ each take at most $m$ values. For each $v\in V(G)$, we either have $\psi_0(v)=0$ or
    \[\varphi(v)=\frac{a_1\psi_1(v)+a_{-1}\psi_{-1}(v)}{2\psi_0(v)}.\]
    In particular, on the set $\{v:\psi_0(v)\neq 0\}$, the function $\varphi$ takes at most distinct $m^3$ values. \cref{lem:1-comp-cor} implies that, for $v\in V(G)$ chosen uniformly at random, $\psi_0(v)=0$ with probability at least $1-m^3/(k-1)$. 
    This holds for any $\psi_0\in V(C)$.

    Now, we may assume that $a_1\neq 0$ and $\psi_1\neq\psi_0$, since our component is not a singleton. 
    Since $\int4\varphi\psi_0\psi_1=a_1\neq 0$, the $2$-product property implies that
    \[2\psi_0\psi_1=a_1\varphi+a'\varphi'\]
    for some basis element $\varphi'\neq\varphi$ and some $a'\in\R$. The function $2\psi_0\psi_1$ takes at most $m^2$ distinct values. Since $\varphi=a_1^{-1}(2\psi_0\psi_1-a'\varphi')$ has image size $k$, the function $a'\varphi'$ takes at least $k/m^2>6$ distinct values. (In particular, $a'\neq 0$.) \cref{lem:1-comp} thus implies that $\cD(\varphi')\sim\cD_n^{\mathrm p}$ for some $n$ and $\mathrm p$. In particular, $\norm{\varphi'}_\infty\leq 1$. This implies that
    \begin{equation}\label{eq:psi-prod-infty}
        \norm{2\psi_0\psi_1}_\infty\leq \abs{a_1}+\abs{a'}.
    \end{equation}

    However, $(2\psi_0\psi_1)(v)\neq0$ with probability at most $m^3/(k-1)$. So, \eqref{eq:psi-prod-infty} implies that
    \[\frac{\abs{a_1}}2=\abs*{\int \varphi(2\psi_0\psi_1)}\leq\frac{m^3}{k-1}\norm{2\psi_0\psi_1}_\infty\leq\frac{m^3}{k-1}(\abs{a_1}+\abs{a'}).\]
    Using $a'=\int 4\varphi'\psi_0\psi_1$, we obtain an analogous inequality with $a_1$ and $a'$ swapped. Adding the two and recalling that $a_1\neq 0$, we have
    \[\frac{\abs{a_1}+\abs{a'}}2\leq \frac{2m^3}{k-1}(\abs{a_1}+\abs{a'})\implies k-1\leq 4m^3,\]
    a contradiction. The result follows.
\end{proof}

\begin{lemma}\label{lem:no-zero-product}
    Let $\varphi\in\cB$ have image size $k\geq 7$, and let $\psi$ be another basis element. 
    Suppose that, for some positive integer $\ell$ and some $b\in\{-1,0,1\}$, we have $T_\ell(\varphi)\psi=b\psi$ identically. Then in fact $T_\ell(\varphi)=b$ identically.
\end{lemma}
\begin{proof}
    We first treat the case $b\in\{-1,1\}$.
    Since $\psi$ is not identically zero, it must be that $T_\ell(\varphi)$ attains the value $b$. 

    Given $T_\ell(\varphi)\psi=b\psi$, we have
    \[b=\int2b\psi^2=\int T_\ell(\varphi)(2\psi^2)=\int T_\ell(\varphi)+\int T_\ell(\varphi)T_2(\psi).\]
    If $\int T_\ell(\varphi)\neq 0$, then the last part of \cref{lem:1-comp} implies that $T_\ell(\varphi)$ is constant, and thus it must be identically $b$. Otherwise, we have $\norm{T_\ell(\varphi)}_{L^2}^2=1/2$ and $\int T_\ell(\varphi)T_2(\psi)=b$.
    Since $T_2(\psi)$ must be identically zero or a scalar multiple of a basis element, we conclude that
    \[2\psi^2-1=T_2(\psi)=2bT_{\ell}(\varphi).\]
    In particular,
    \[(2bT_\ell(\varphi)-2b^2)\psi=(2\psi^2-3)\psi.\]
    So, for every $v\in V(G)$, we have $\psi(v)\in\{0,\sqrt{3/2},-\sqrt{3/2}\}$. Let $n_0$ (resp.\ $n_+,n_-$) be the number of $v\in V(G)$ with $\psi(v)=0$ (resp.\ $\psi(v)=\sqrt{3/2},\psi(v)=-\sqrt{3/2}$). Since $\ip{\psi,\1}=0$, we have $n_+=n_-$, and since $\ip{\psi,\psi}=1/2$ we have
    \[\frac12\abs{V(G)}=\frac32(n_++n_-).\]
    This implies $6n_+=\abs{V(G)}$, contradicting the oddness of $\abs{V(G)}$. We conclude that in fact $T_\ell(\varphi)=b$ identically.

    Finally, we treat the case $b=0$. In this case, the identity $T_\ell(\varphi)\psi=0$ implies $T_{2\ell}(\varphi)\psi=-\psi$. This implies $T_{2\ell}(\varphi)=-1$ identically, and so $T_\ell(\varphi)=0$ identically.
\end{proof}

\begin{lemma}\label{lem:large-comp-spec}
    Each connected component of $H_\varphi$ has at least $k/3$ vertices.
\end{lemma}
\begin{proof}
    Suppose a component of $H_\varphi$ has $\ell$ vertices $\psi_1,\ldots,\psi_\ell$, where $\ell<k/3$.
    For any polynomial $P$, the function $P(\varphi)\psi_1$ is a linear combination of $\psi_1,\ldots,\psi_\ell$. In particular, the $\ell+1$ functions $\{T_j(\varphi)\psi_1:0\leq j\leq\ell\}$ are linearly dependent.
    If any of these functions are identically zero, then \cref{lem:no-zero-product} gives that $T_j(\varphi)=0$ identically for some $0\leq j\leq\ell$, which cannot occur since $\ell<k$.    
    We conclude (1) that there exists some polynomial $Q$ of degree at most $\ell$ for which $Q(\varphi)\psi_1=0$, and (2) that the inner product
    \begin{equation}\label{eq:ip-TiTj}
        \ip{T_i(\varphi)\psi_1,T_j(\varphi)\psi_1}
    \end{equation}
    is nonzero for some $0\leq i<j\leq\ell$. 
    
    We next expand the inner products in \eqref{eq:ip-TiTj}. For $0\leq i<j\leq\ell$,
    \begin{equation}\label{eq:small-comp-ip}
    4\int(T_i(\varphi)\psi_1)(T_j(\varphi)\psi_1)=\int(T_{i+j}(\varphi)+T_{j-i}(\varphi))(2\psi_1^2).
    \end{equation}
    Since $1\leq i+j<2\ell\leq k-1$, \cref{lem:1-comp} implies that $T_{j-i}(\varphi)$ and $T_{i+j}(\varphi)$ are basis elements. So, if \eqref{eq:small-comp-ip} is nonzero for some $0\leq i<j\leq \ell$, then either $T_{i+j}(\varphi)$ or $T_{i-j}(\varphi)$ is a component of the decomposition of $2\psi_1^2-1$. We conclude that, for some $1\leq m\leq 2\ell-1$ and some nonzero $c\in\R$, we have
    \[2\psi_1^2=1+cT_m(\varphi).\]
    We now use (1). We have
    \[0=2Q(\varphi)\psi_1^2=Q(\varphi)(1+cT_m(\varphi)).\]
    We conclude that $\varphi$ takes at most $\deg Q+m\leq 3\ell-1$ values. This contradicts the assumption that $\ell< k/3$.
\end{proof}

\cref{lem:large-comp} in the graph case now follows from combining \cref{lem:varied-or-large,lem:large-comp-spec}.

We conclude this section with an additional observation, in the spirit of the contents of this section, which will later enable us to rule out some potential components of $H_\varphi$.

\begin{lemma}\label{lem:T2-not-orth}
    Let $\varphi,\psi\in\mathcal B$, and suppose that $T_2\varphi\not\perp T_2\psi$ and $\abs{\im\varphi}\geq 14$. Then $T_2\varphi=T_2\psi$ or $T_2\varphi=-T_2\psi$.
\end{lemma}
\begin{proof}
    Since $\varphi,\psi\in\mathcal B$, the fact that $\cB$ is a $2$-product eigenbasis implies that both $T_2\varphi$ and $T_2\psi$ are scalar multiples of elements of $\cB$. 
    In particular, for some constant $c\neq 0$, we have $T_2\psi=cT_2\varphi$. 
    Since $\abs{\im\varphi}\geq 14$, we have $\abs{\im(T_2\varphi)}\geq 7$, and so $\abs{\im\psi}\geq 7$. We conclude from \cref{lem:1-comp} that $T_2\varphi$ and $T_2\psi$ both have $L^2$ norm $1/\sqrt2$. This implies $c=\pm 1$, as desired.
\end{proof}

\subsection{Identifying the components}

Here we prove \cref{lem:primitive,lem:comp-labels}.

\begin{proof}[Proof of \cref{lem:primitive}]
    Take $\varphi=\varphi_0\in\cB$.
    For each $i\geq 0$, let $\varphi_{i+1}\in\cB$ be such that $\R[\varphi_i]\subsetneq\R[\varphi_{i+1}]$; if no such $\varphi_{i+1}$ exists, then we have found some primitive $\varphi':=\varphi_i$ for which $\varphi\in\R[\varphi']$.

    If we are in the graph setting, then the finite-dimensionality of each $\R[\varphi_i]$ is enough to prove that the sequence $\{\varphi_i\}$ eventually terminates in a primitive eigenfunction.
    In the manifold setting, by \cref{lem:1-comp}, we can find some $m_i\in\N$ such that $\varphi_i=\pm T_{m_i}(\varphi_{i+1})$ for some $m_0,m_1,\ldots>1$.
    Since $M$ is compact, $M$ may be covered by finitely many connected sets $S_1,\ldots,S_n$ such that $\varphi(S_i)$ is an interval of length at most $1$ for each $1\leq i\leq n$. 

    Given a positive integer $p$, let $g_p$ be the largest gap between the elements of $\{\cos(\pi r/p):r\in\Z\}$. As $p$ grows, $g_p$ tends to zero. Let $q$ be such that $g_p<1/(4n)$ for every $p>q$. Any sub-interval of $[-1,1]$ with length $1/(2n)$ will contain, for each $p>q$, some two values $\cos(\pi r/p)$ and $\cos(\pi (r+1)/p)$ for integer $r$. Now, select $k$ such that $p:=m_0\cdots m_{k-1}>q$. On any sub-interval of $[-1,1]$ with length $1/(2n)$, the function $T_p(x)$ takes all values in $[-1,1]$.
    
    We also have $T_{p}(\varphi_k)=\pm\varphi_0$.
    For each $i$, the image $\varphi_k(S_i)$ is some interval in $[-1,1]$ by \cref{lem:1-comp}. 
    (We are using that the image of $\varphi_k$ is exactly the support of the probability distribution $\cD(\varphi_k)$. Indeed, for any $x\in M$, any small ball in $M$ around $x$ has positive volume and is mapped near $\varphi_k(x)$ by $\varphi_k$, and so $\cD(\varphi_k)$ assigns positive mass to some neighborhood of $\varphi_k(x)$.)
    If this interval has length at least $1/(2n)$, then it contains both $\cos(\pi r/p)$ and $\cos(\pi(r+1)/p)$ for some integer $r$. This implies that $\varphi_0(S_i)=\pm T_p(\varphi_k(S_i))$ contains both $1$ and $-1$, which contradicts our definition of $S_i$. Now,
    \[\varphi_k(M)=\bigcup_{i=1}^n\varphi_k(S_i)\]
    is an interval which is a union of $n$ intervals of length at most $1/(2n)$. In particular, it is an interval of length at most $1$. This contradicts \cref{lem:1-comp} and finishes the proof.
\end{proof}

Note that, in the manifold case, the same result above may also be obtained by proving that if a nonconstant polynomial of an eigenfunction is itself an eigenfunction, then its eigenvalue is at least that of the original eigenfunction. This is immediate after a Rayleigh quotient computation and an integration by parts.

We next prove \cref{lem:comp-labels}. 
The following simple lemma removes much of the complexity from the classification of the components of $H_\varphi$ and is the main reason we specialize to primitive eigenfunctions.

\begin{lemma}\label{lem:no-loop}
    Suppose that $\varphi$ is primitive. Then no component of $H_\varphi$, except possibly for the $\1$-component, has a self-loop.
\end{lemma}
\begin{proof}
    Suppose that some vertex $\psi$ of $H_\varphi$ has a self-loop. We have $\int2\varphi\psi^2\neq 0$. So, the basis decomposition of $2\psi^2$ consists of $\1$ (since $\norm{\psi}_{L^2}^2=1/2$) and some scalar multiple $c\varphi$ of $\varphi$. We conclude that
    \[T_2(\psi)=2\psi^2-\1=c\varphi.\]
    We conclude from the assumption that $\varphi$ is primitive that $\psi\in\R[\varphi]$; that is, $\psi$ is in the $\1$-component of $H_\varphi$.
\end{proof}

Combining \cref{lem:no-loop} with \cref{obs:max-deg-2,lem:large-comp-spec} implies that, if $\abs{\im\varphi}\geq 7$, then every component of $H_\varphi$ is a path without self-loops (which may be finite, uni-infinite, or bi-infinite) or a finite cycle.
We will describe the possibilities for such components first ``locally'' and then ``globally'' (the latter will only be relevant in the graph setting).

For a positive integer $m$, we say a vertex $\psi\in\cB$ is \emph{$m$-internal} if all leaves of $H_\varphi$ have distance at least $m$ from $\psi$. (A $1$-internal vertex is exactly a non-leaf, and a $2$-internal vertex is exactly a vertex neither of whose neighbors are leaves, for example.)  

\begin{lemma}\label{lem:doubly-internal}
    Suppose $k:=\abs{\im\varphi}\geq 15$. If $\psi$ is a $2$-internal vertex of $H_\varphi$ which lies outside the $\1$-component, then the squares of the labels of the edges incident to $\psi$ sum to $2$.    
\end{lemma}
\begin{proof}
    Write $\psi_0:=\psi$ and let $\psi_{-2},\psi_{-1},\psi_0,\psi_1,\psi_2$ be vertices in $H_\varphi$ connected in that order, in such a way that
    \begin{align*}
        2\varphi\psi_{-1}&=a_{-1}\psi_0+a_{-2}\psi_{-2}\\
        2\varphi\psi_0&=a_0\psi_1+a_{-1}\psi_{-1}\\
        2\varphi\psi_1&=a_1\psi_2+a_0\psi_0.
    \end{align*}
    Since $\abs{\im\varphi}\geq 15$, \cref{lem:large-comp} implies that the component of $H_\varphi$ containing $\psi_0$ has cardinality at least $5$. In particular, $\psi_{-2},\ldots,\psi_2$ are distinct.
    Now, we may compute
    \[2T_2(\varphi)\psi_0=a_{-1}a_{-2}\psi_{-2}+(a_{-1}^2+a_0^2-2)\psi_0+a_0a_1\psi_2.\]
    Since $T_2(\varphi)$ is a nonzero scalar multiple of a basis element by our normalization $\norm{\varphi}_{L^2}^2=1/2$ and each $a_i$ is nonzero, the $2$-product property implies $a_{-1}^2+a_0^2=2$, as desired.
\end{proof}

\begin{lemma}\label{lem:triply-internal}
    Suppose $k:=\abs{\im\varphi}\geq 21$. If $\psi$ is a $3$-internal vertex of $H_\varphi$ which lies outside the $\1$-component, then the labels of the edges incident to $\psi$ are both $\pm 1$.
\end{lemma}
\begin{proof}
    Adopt the notation in the previous lemma and define $\psi_{-3},\psi_3,a_{-3},a_2$ accordingly. 
    We can compute
    \begin{align*}
    2T_3(\varphi)\psi_0
    &=a_{-2}a_{-1}(2\varphi\psi_{-2})-2\varphi\psi_0+a_0a_1(2\varphi\psi_2)\\
    &=a_{-3}a_{-2}a_{-1}\psi_{-3}+a_{-1}(a_{-2}^2-1)\psi_{-1}+a_0(a_1^2-1)\psi_1+a_0a_1a_2\psi_3.
    \end{align*}
    The property $k\geq 21$ is enough to imply, using \cref{lem:large-comp}, that $\{\psi_{-3},\ldots,\psi_3\}$ are all distinct.
    We thus have $a_{-2}^2=a_1^2=1$. 
    \cref{lem:doubly-internal} applied to $\psi_{-1}$ and $\psi_0$ gives the result.
\end{proof}

\begin{lemma}\label{lem:leaf}
    Let $\varphi$ be a primitive eigenfunction with $k:=\abs{\im\varphi}\geq 21$.
    Suppose that $\psi$, not in the $\1$-component of $H_\varphi$, is a leaf of $H_\varphi$ without a self-loop, and let $\psi_1$ be the unique neighbor of $\psi$ in $H_\varphi$.
    Then the label of the edge $(\psi,\psi_1)$ is in the set $\{-\sqrt2,-1,1,\sqrt2\}$.
    If this label is $\pm 1$, then $T_2\psi=-T_2\varphi$, while if this label is $\pm\sqrt2$ then $T_2\psi_1=T_2\varphi$.
\end{lemma}
\begin{proof}
    Let $a_0$ be the edge label in question, so that $2\varphi\psi=a_0\psi_1$.
    We have
    \[\frac{a_0^2}2=\int (a_0\psi_1)^2=\int(2\varphi\psi)^2=\int(1+T_2\varphi)(1+T_2\psi)=1+\int T_2\varphi\cdot T_2\psi.\]
    Now, let $\psi_2$ and $\psi_3$ be the next two vertices along the path $\psi,\psi_1,\ldots$. 
    By \cref{lem:large-comp}, the component of $H_\varphi$ has at least $7$ vertices. In particular, $\psi_3$ is $3$-internal. Applying \cref{lem:triply-internal} to $\psi_3$ and then \cref{lem:doubly-internal} to $\psi_2$, the edge connecting $\psi_1$ and $\psi_2$ has label $\pm 1$. By swapping the sign of $\psi_2$, we can assume this edge has label $1$. Therefore we have
    $2\varphi\psi_1=a_0\psi_0+\psi_2$.
    This enables us to compute
    \[\frac{a_0^2+1}2=\int(a_0\psi+\psi_2)^2=\int(2\varphi\psi_1)^2=\int(1+T_2\varphi)(1+T_2\psi_1)=1+\int T_2\varphi\cdot T_2\psi_1.\]
    Finally, \cref{lem:T2-not-orth} implies that both $2\int T_2\varphi\cdot T_2\psi$ and $2\int T_2\varphi\cdot T_2\psi_1$ are elements of $\{-1,0,1\}$. We conclude that $a_0^2\in\{1,2,3\}\cap\{0,1,2\}=\{1,2\}$, as desired.
\end{proof}

\begin{proofwithclaims}{lem:comp-labels}
    We first treat the manifold case.
    By \cref{lem:large-comp}, all components of $H_\varphi$ are infinite, and so by \cref{obs:max-deg-2,lem:no-loop} all components are either bi-infinite paths or uni-infinite paths with no self-loop. 
    Every vertex of a bi-infinite path is $3$-internal, and so \cref{lem:triply-internal} implies that all labels of such a path are $\pm 1$.
    For a uni-infinite path which is not the $\1$-component, label the vertices $\psi_0,\psi_1,\ldots$ along the path. The edges $(\psi_i,\psi_{i+1})$ for $i\geq 2$ have labels $\pm 1$ by \cref{lem:triply-internal}. Using this fact for $i=2$ and \cref{lem:doubly-internal}, the edge $(\psi_1,\psi_2)$ has label $1$ as well. 
    Finally, \cref{lem:leaf} implies that the label $a_0$ of the edge $(\psi_0,\psi_1)$ squares to $1$ or $2$. If $a_0^2=2$ then $T_2\psi_1=T_2\varphi$, which implies $\varphi^2=\psi_1^2$. We obtain that $\varphi=\psi_1$ or $\varphi=-\psi_1$ identically, which cannot occur since we have assumed the component in question is not the $\1$-component. If $a_0^2=1$ then $T_2\psi_0=-T_2\varphi$. 
    This can only occur for one component: otherwise, we would have $\psi_0^2=(\psi_0')^2$ for distinct $\psi_0,\psi_0'\in\cB$, which cannot occur.
    We conclude the result in the manifold setting.

    We now come to the graph setting. By \cref{obs:max-deg-2,lem:large-comp,lem:no-loop}, any component must be a cycle or a finite path with no self-loops.
    \begin{itemize}
        \item In the former case, all edge labels are $\pm1$ by \cref{lem:triply-internal}. 
        Negating any individual vertex negates both of the edge labels incident to it. 
        There is some sequence of such operations which ensures that all edge labels are $1$, or that all edge labels are $1$ except one which is $-1$.
    
        \item In the latter case, as in the manifold-setting argument in the previous case, all but possibly the two edges incident to the leaves have labels $\pm1$. (Here we are using \cref{lem:large-comp} and the fact that $\abs{\im\varphi}$ is large enough to obtain that every component has a $3$-internal vertex.) The two remaining edges have labels in $\{\pm1,\pm\sqrt2\}$.
        By negating vertices if necessary, we can enforce that all edge labels are positive.
    \end{itemize}
    
    We thus have five types of potential components:
    \begin{enumerate}[(I)]
        \item cycles with all edge labels $1$,
        
        \item cycles with all edge labels $1$ except one $-1$,
        
        \item paths with all labels $1$, 
        
        \item paths with all labels $1$ except one $\sqrt2$, 
        
        \item and paths with all labels $1$ except two $\sqrt2$.
    \end{enumerate}

    Our next step is to constrain the lengths of such components using the classification of potential $\varphi$ provided by \cref{lem:1-comp} and using \cref{lem:no-zero-product}.
    Enumerate the vertices of some component of size $\ell$ as $\psi_0,\ldots,\psi_{\ell-1}$. Multiplication by $\varphi$ acts as a linear transformation on space $\R\psi_0+\cdots+\R\psi_{\ell-1}$.
    If the characteristic polynomial of this action is $Q(\varphi)$, we have $Q(\varphi)\psi_i=0$ for each $0\leq i<\ell$. We may explicitly compute the polynomials $Q$ for each potential component, obtaining
    \begin{align}
        Q(\varphi)&=T_\ell(\varphi)-1,\tag{in case (I)}\\
        Q(\varphi)&=T_\ell(\varphi)+1,\tag{in case (II)}\\
        Q(\varphi)&=U_\ell(\varphi),\tag{in case (III)}\\
        Q(\varphi)&=T_\ell(\varphi),\tag{in case (IV)}\\
        Q(\varphi)&=T_\ell(\varphi)-T_{\ell-2}(\varphi)\tag{in case (V)},
    \end{align}
    where $U_\ell$ is the degree-$\ell$ Chebyshev polynomial of the second kind.
    We note the polynomial divisibilities
    \[U_\ell\mid T_{2\ell+2}-1,\qquad T_\ell-T_{\ell-2}\mid T_{2\ell-2}-1.\]
    Now, \cref{lem:no-zero-product} allows us to upgrade a statement of the form ``$T_n(\varphi)\psi_i=b\psi_i$'' for some positive integer $n$ and some $b\in\{-1,0,1\}$ to a statement of the form ``$T_n(\varphi)=b$ identically.'' 

    We can also upper-bound the lengths of potential components:
    \begin{claim}\label{cl:comp-small}
        In the cycle case, we have $\ell\leq 2k$. In the path case, we have $\ell\leq k$.
    \end{claim}
    \begin{proofofclaim}
        We first treat the cycle case. By induction on $i$, each $\psi_i$ lies in $\R[\varphi]\psi_0+\R[\varphi]\psi_1$. Since the $\psi_i$ are linearly independent and $\dim\R[\varphi]=k$, we conclude that $\ell\leq 2k$.

        In the path case, each $\psi_i$ lies in $\R[\varphi]\psi_0$, and so we conclude $\ell\leq k$ for the same reason.
    \end{proofofclaim}

    Let $n$ and $b$ be such that $T_n(\varphi)=b$ identically. In cases (I)--(V), respectively, we have $(n,b)=(\ell,1),(\ell,-1),(2\ell+2,1),(\ell,0),(2\ell-2,1)$. By \cref{cl:comp-small} we have $n\leq 2k+2$.
    
    Also, let $m,\mathrm p$ be such that $\cD(\varphi)=\cD_m^{\mathrm p}$, so that $m\in\{2k-1,2k\}$ and $\mathrm p\in\{\mathrm{even},\mathrm{odd}\}$, given by \cref{lem:1-comp}.

    \begin{enumerate}[(a)]
        \item Suppose $m=2k$ and $\mathrm p=\mathrm{odd}$.
        There exists some $v\in V(G)$ for which $\varphi(v)=\cos(\pi/(2k))$. Therefore $b=T_n(\varphi(v))=\cos(\pi n/(2k))$. 
        We conclude that either $n=k$ and $b=0$ or $n=2k$ and $b=-1$. 
        This allows for type (II) components of size $\ell=2k$ or type (IV) components of size $\ell=k$.

        \item Suppose $m=2k-1$ and $\mathrm p=\mathrm{even}$. 
        There exists some $v\in V(G)$ for which $\varphi(v)=\cos(2\pi/(2k-1))$.
        Therefore $b=T_n(\varphi(v))=\cos(2\pi n/(2k-1))$. 
        We conclude that $n=2k-1$ and $b=1$. 
        This allows for type (I) components of size $\ell=2k-1$.

        \item Suppose $m=2k-1$ and $\mathrm p=\mathrm{odd}$. 
        There exists some $v\in V(G)$ for which $\varphi(v)=\cos(\pi/(2k-1))$.
        Therefore $b=T_n(\varphi(v))=\cos(\pi n/(2k-1))$. 
        We conclude that $n=2k-1$ and $b=-1$. 
        This allows for type (II) components of size $\ell=2k-1$.
    \end{enumerate}

    In cases (b) or (c), \cref{obs:div} implies $2k-1\mid \abs{V(G)}$. 
    In particular, $H_\varphi$ possesses some component distinct from the $\1$-component whose order is not a multiple of $2k-1$.
    Therefore, these cases cannot occur.
    We conclude that $m=2k$ and $\mathrm p=\mathrm{odd}$.
    Using \cref{obs:div}, this implies that $k$ is odd.
    What remains is to exclude the type (IV) components of size $\ell=k$.
    Any such component has some edge label $\sqrt2$, and so \cref{lem:leaf} implies that (up to relabeling by writing our path in the opposite order) we have $\psi_1^2=\varphi^2$. We obtain
    \[0=2\varphi^2\psi_0^2-\psi_1^2=\varphi^2(2\psi_0^2-1).\]
    We conclude that $\varphi\cdot T_2(\psi_0)=0$. 

    Now, $2\psi_0^2-1$ is either identically zero or a scalar multiple of a basis element. If it is identically zero, then $\int\psi_0=0$ implies that $\cD(\psi_0)$ is uniformly distributed between $1/\sqrt2$ and $-1/\sqrt2$. In particular, this implies that $\abs{V(G)}$ is even, a contradiction.

    On the other hand, if $2\psi_0^2-1$ is a scalar multiple of a nontrivial basis element, then \cref{lem:no-zero-product} with $\ell=1$ and $b=0$ implies that $\varphi=0$. The result follows.
\end{proofwithclaims}

\vspace{3em}

\appendix

\section{Reconstruction Theorems}
\label{sec: reconstruction}

We now present proofs of Theorems \ref{thm: recover M from triple prod} and \ref{thm: graph reconstruction}. We treat the manifold case first, relying on the algebraic perspective developed in  detail in \cite{Nestruev2003}. The argument proceeds in three stages: from the given data we successively recover the ring $C^\infty(M)$, the topological and smooth structure of $M$, and finally the metric. We then explain the reconstruction of weighted graphs from eigenvalues and triple product constants, which, to the best of our knowledge, has never appeared in print.

To address the manifold case, we fix an orthonormal basis $\{\phi_k\}_{k=0}^\infty$ of $L^2(M)$ consisting of Laplace eigenfunctions with eigenvalues $0=\lambda_0<\lambda_1\leq \dots\to \infty$.

\begin{proof}[Proof of \cref{thm: recover M from triple prod}]
     Extract the dimension $d$ of $M$ from the spectrum using the Weyl law. Recall that the Sobolev spaces on $M$ can be described as
    \[
        H^s(M)=\left\{\sum_{k=0}^\infty a_k\phi_k\st \sum_{k=0}^\infty (1+k)^{2s/d}a_k^2<\infty\right\},
    \]
and that the space of smooth functions is given by
    \[
        C^\infty(M)=\bigcap_{s\in \R}H^s(M).
    \]
    Smooth functions are therefore characterized in terms of the rapid decay of their Fourier coefficients. Furthermore, on the Fourier side, multiplication on the space of smooth functions is described precisely by the triple product coefficients. We conclude that, from the spectrum and the triple product constants, we recover $C^\infty(M)$ as an algebra.

    From $C^\infty(M)$, we can reconstruct the manifold $M$. This is the content of Chapter 7 of \cite{Nestruev2003}, specifically of Theorems 7.2 and 7.7. The main idea is that maximal ideals of $C^\infty(M)$ are precisely the sets of functions vanishing at a fixed point of $M$. All that remains is therefore to recover the metric.

    To that end, recall that the map
    \[
        (\phi_1,\dots, \phi_n)\colon M\to \R^n
    \]
    is an embedding when $n$ is sufficiently large \cite{BerardBessonGallot1994}. Observe
    \begin{equation}
    \label{eq: recover metric}
        \langle\grad \phi_j,\grad \phi_k\rangle_g=\frac{1}{2}\big(\phi_j\Delta \phi_k+\phi_k\Delta \phi_j-\Delta(\phi_j\phi_k)\big).
    \end{equation}
    The right hand side is computable from the triple product constants and the eigenvalues. Since the differentials $d\phi_k(p)$, $1\leq k\leq n$ span $T_p^*M$, Equation \eqref{eq: recover metric} in fact determines $g^*_p$, and hence $g_p$ on $T_pM$. This recovers $g$ pointwise, completing the proof.
\end{proof}

We now turn to the case of weighted graphs. Let $G=(V,E,w)$ be a simple weighted graph on $|V|=n$ vertices, where $w\st E\to (0,\infty)$. The weighted graph Laplacian $L=L_G$ acts on $\R^V$ by
\[
    (Lf)(u)=\sum_{v\sim u}w(uv)\big(f(u)-f(v)\big),
\]
and is symmetric and positive semidefinite. Its eigenvalues $0=\lambda_1\leq \cdots\leq \lambda_n$ together with an orthogonal eigenbasis $v_1,\dots,v_n\in \R^V$ (with respect to the $L^2$ inner product) give rise to triple product constants
\[
    c_{ijk}=\frac1n\sum_{\ell\in V}(v_i)_\ell(v_j)_\ell(v_k)_\ell.
\]

\begin{proof}[Proof of \cref{thm: graph reconstruction}]
    We can assume that our eigenvectors $v_1,\ldots,v_n$ are normalized with respect to the $L^2$ norm.
	Consider the $\R$-algebra $A$ generated by $v_1,\ldots,v_n$ with multiplication
    \[
        v_iv_j=\sum_kc_{ijk}v_k.
    \]
    This algebra is isomorphic to $\R^V$ via some map $f\colon A\to\R^V$. 
    The set of idempotents of $\R^V$ is $\{0,1\}^{V}$, and so the set $S$ of idempotents of $A$ has cardinality exactly $2^{\abs{V}}$. 
    Let $W$ be the set of idempotents $w$ of $A$, distinct from $0$, which satisfy $wS\subset\{0,w\}$. 
    The elements $f(W)$ of $\R^V$ are exactly the vectors which assign $1$ to some vertex of $V$ and $0$ elsewhere. Therefore, $W$ forms a basis of $A$ as an $\R$-vector space. Enumerate $W=\{w_1,\ldots,w_n\}$, and let $b_{ij}$ be such that $v_i=\sum_{j=1}^nb_{ij}w_j$. 
    
    The vectors $(b_{ij})_{j=1}^n\in\R^n$ are the eigenvectors of the Laplacian of $G$, and so we can recover the Laplacian of $G$ as $n^{-1}B^\intercal\Lambda B$ where $\Lambda=\operatorname{diag}(\lambda_1,\ldots,\lambda_n)$.
\end{proof}

\bibliography{bib.bib}{}
\bibliographystyle{plain}

\end{document}